\documentclass[12pt,notitlepage, intlimits]{amsart}
\usepackage{latexsym,amsthm, amsmath, amssymb,amsfonts} 
\newcommand{\mr}{\ensuremath{\mathbb R}}

\newcommand{\shortmod}{\ensuremath{\negthickspace \negthickspace \negthickspace \pmod}}

\newcommand{\half}{\ensuremath{ \frac{1}{2}}}

\newcommand{\sumstar}{\sideset{}{^*}\sum}

\newcommand{\thalf}{\tfrac12}

\theoremstyle{plain}		
	\newtheorem{mytheo}{Theorem}[section]
	
	\newtheorem{myprop}[mytheo]{Proposition}
	\newtheorem{mycoro}[mytheo]{Corollary}
     \newtheorem{mylemma}[mytheo]{Lemma}
	
	\newtheorem{myconj}[mytheo]{Conjecture}

\theoremstyle{remark}

\begin{document}
\today
\title{The twisted fourth moment of the Riemann zeta function}
\thanks{Both authors were supported by the American Institute of Mathematics and the second author was supported by an NSF postdoctoral fellowship.}
\author{C.P. Hughes}
\address{Department of Mathematics \\
    University of York \\
    York \\
    YO10 5DD \\
    U.K.}
 \email{ch540@york.ac.uk}

\author{Matthew P. Young}
\address{American Institute of Mathematics \\
    360 Portage Ave.\\
    Palo Alto, CA 94306-2244 \\
    U.S.A.}
\curraddr{Department of Mathematics \\
	  Texas A\&M University \\
	  College Station \\
	  TX 77843-3368 \\
		U.S.A.}
\email{myoung@math.tamu.edu}

\begin{abstract}
We compute the asymptotics of the fourth moment of the Riemann
zeta function times an arbitrary Dirichlet polynomial of length
$T^{\frac{1}{11} - \varepsilon}$.
\end{abstract}
\maketitle
\section{Introduction}

The study of the moments of the Riemann zeta function has a long
and distinguished history, starting with the work of Hardy and
Littlewood in 1918 and continuing to the present day. One motivation
for understanding moments is that they yield
information about the maximum size of the zeta function (the
Lindel\"of Hypothesis); another application is to zero density estimates
which in turn have consequences for primes in short intervals.
However they have become an interesting
topic in their own right. Very few rigorous results are known,
just the second and fourth power moments. Indeed, it is only
recently that a believable conjecture for higher powers has been made.

The twisted moments (that is, moments of the Riemann zeta function
times an arbitrary Dirichlet polynomial) are important too, for
example Levinson's method of detecting zeros of the zeta function
lying on the critical line requires knowing the asymptotics of the
mollified second moment.  In a series of papers, Duke, Friedlander,
and Iwaniec used estimates for amplified moments of a family of
$L$-functions in order to deduce a subconvexity bound for an individual
member of the family.  Of course, there are far easier methods to
give a subconvexity bound for zeta, but there are close analogies between
different families and it is desirable to understand the structure of
these amplified moments in general.

In this paper, we prove an asymptotic formula for the twisted fourth moment of
the Riemann zeta function, where we may take a Dirichlet
polynomial of length up to $T^{\frac{1}{11} - \varepsilon}$.

\subsection{Previous results}

The first result concerning an asymptotic expansion for the second
moment of the Riemann zeta function is due to Hardy and Littlewood
in 1918 \cite{HardyLittle18} where they showed that
\begin{equation}
\int_0^T \left|\zeta(\tfrac12+i t)\right|^{2} \; dt \sim T \log T.
\end{equation}

In 1926 Ingham \cite{Ingham26} improved this by calculating the
lower order terms via finding the asymptotics of the shifted
second moment. He showed that if $|\Re(\alpha)| \leq 1/2-\varepsilon$
and $|\Re(\beta)| \leq 1/2-\varepsilon$ for any fixed $\varepsilon>0$ then
\begin{multline}
\int_0^T \zeta(\tfrac12+\alpha+it) \zeta(\tfrac12+\beta-it) \; dt
= \int_0^T \left(\zeta(1+\alpha+\beta) +
\left(\frac{t}{2\pi}\right)^{-\alpha-\beta}
\zeta(1-\alpha-\beta)\right)\; dt \\
+ O\left(T^{1/2-\Re(\alpha+\beta)/2}\log T\right),
\end{multline}
with the error term uniform in $\alpha$ and $\beta$. From this he
deduced that
\begin{equation}
\int_0^T \left|\zeta(\tfrac12+i t)\right|^{2} \; dt = T \left(\log
\frac{T}{2\pi} + 2\gamma - 1 \right) + O(T^{1/2}\log T).
\end{equation}

The final result concerning the second moment we wish to highlight
is due to Balasubramanian, Conrey and Heath-Brown \cite{BCHB85}.
They show that for $\theta=1/2-\varepsilon$, if
\begin{equation}
M(s) = \sum_{h\leq T^{\theta}} \frac{a(h)}{h^s},
\end{equation}
with $a(h) \ll h^\epsilon$, then
\begin{equation}
\int_0^T \left|\zeta(\tfrac12+it)\right|^2
\left|M(\tfrac12+it)\right|^2 \;dt \sim T \sum_{h,k \leq T^\theta}
\frac{a(h) \overline{a(k)} (h,k)}{h k} \left(\log\left(\frac{T
(h,k)^2}{2\pi h k } \right) + 2\gamma - 1\right),
\end{equation}
where $(h,k)$ denotes the greatest common divisor of $h$ and $k$.
Conrey \cite{Con89} increased the length of the polynomial to
$T^{4/7-\varepsilon}$ in the case when the coefficients $A(s)$ had
a specific form similar to the M\"obius function. This allowed him
to use Levinson's method \cite{Lev} to show that more than $40\%$
of the zeros of the zeta function lie on the critical line,
improving on Levinson's original estimate of at least $33\%$ of
zeros satisfy the Riemann Hypothesis.

The fourth moment of the Riemann zeta function is much more
complicated. The first asymptotic result is due to Ingham
\cite{Ingham26} who showed that
\begin{equation}
\int_0^T \left|\zeta(\tfrac12+it)\right|^4 \; dt =
\frac{1}{2\pi^2} T (\log T)^4 + O\left(T (\log T)^3\right).
\end{equation}
Though published in 1926, the result was first announced in 1923.
During the interceding time Titchmarsh \cite{Tit28_a} discovered a
similar result using completely different methods from which
Ingham's result follows via a Tauberian theorem. Titchmarsh's
paper was published in the same journal as Ingham's paper. He
showed that
\begin{equation}
\int_0^\infty \left|\zeta(\tfrac12+it)\right|^4 e^{-t/T} \;dt \sim
\frac{1}{2\pi^2} T \left(\log T \right)^4.
\end{equation}

Atkinson \cite{Atk41} found the lower order terms (as a degree
four polynomial in $\log T$) for fourth moment in Titchmarsh's
smoothed form, but this could not be shown to imply anything for
the lower order terms in the unsmoothed case. It wasn't until 1979
that Heath-Brown \cite{HB79} managed to calculate the lower order
terms as a degree four polynomial in $\log T$ for the unsmoothed
fourth moment of zeta. His result, when written out in full, is very lengthy. Conrey
\cite{Con96} simplified the answer slightly, by showing that the
polynomial is the residue of a certain function at $s=1$.

Upper bounds on the twisted fourth moment of zeta 
\begin{equation}
\int_0^T \left|\zeta(\tfrac12+it)\right|^4
\left|M(\tfrac12+it)\right|^2 \;dt
\end{equation}
were considered by Deshouillers and Iwaniec \cite{DI} in the case when the Dirichlet polynomial had length $T^{1/5-\varepsilon}$. This result was later improved by Watt \cite{Watt} to polynomials of length up to  $T^{1/4-\varepsilon}$. Both of these results used the spectral theory of automorphic forms.  In his PhD thesis, Jose Gaggero Jara \cite{Jose} calculated an asymptotic result for the twisted fourth moment of zeta
when the length of the Dirichlet polynomial is $T^{4/589 - \varepsilon}$. His result however does not yield all the lower order terms (for example, it does not recover Heath-Brown's result mentioned above).

The shifted fourth moment with no twisting has been studied by Motohashi \cite{Moto} who proved an exact result when a smoothed average was taken. Motivated by the structure of the answer he found, and previous work of Keating and Snaith \cite{KS}, Conrey, Farmer, Keating, Rubinstein and Snaith \cite{CFKRS05} found a heuristic argument (a recipe) to calculate the shifted $2\ell$\textsuperscript{th} moment for any positive integer $\ell$. From such all lower order terms can be calculated, at least conjecturally. At the end of this paper, we will show how the recipe can be modified to cope with shifted twisted $2\ell$\textsuperscript{th} moments.

Independently of us, Motohashi has recently extended his method to handle the twisted case \cite{MotoHecke}.  His method is quite different than ours here because he uses spectral theory to handle the binary divisor problem, whereas we lifted the result of \cite{DFI}, which uses the Weil bound for Kloosterman sums.  Furthermore, our primary goal here is to understand the main terms, while Motohashi devotes his attention to the development of the sum of Kloosterman sums (which do not contribute to the main terms, at least with variables in the ranges of summation restricted by an appropriate approximate functional equation).

\subsection{Results}
Let $M(s)$ be an arbitrary Dirichlet polynomimal of length $T^{\theta}$ given by
\begin{equation}
M(s) = \sum_{h \leq T^{\theta}} \frac{a(h)}{h^{s}}.
\end{equation}
The primary purpose of this paper is to compute the asymptotic behavior of 
\begin{equation}
\label{eq:mollfiedsharpcutoff}
\int_{0}^{T} 
|M(\thalf + it)|^2
|\zeta(\thalf + it)|^4 dt,
\end{equation}
as well as similar integrals with derivatives of $\zeta$ taking the place of $\zeta$.  Since the integrand is nonnegative we can bound the above expression from above and below by a smoothed integral.  By expanding out the Dirichlet polynomial $M(s)$, the problem reduces to the study of the ``twisted" fourth moment of $\zeta$
\begin{equation}
I(h,k) = \int_{-\infty}^{\infty} 
\left(\frac{h}{k}\right)^{-it}
\zeta({\textstyle \half + \alpha + it}) 
\zeta({\textstyle \half + \beta + it}) 
\zeta({\textstyle \half + \gamma - it}) 
\zeta({\textstyle \half + \delta - it}) w(t)dt,
\end{equation}
where $\alpha, \beta, \gamma, \delta$ are sufficiently small complex numbers and $w$ is a nice smooth function.  Since $M(s)$ is arbitrary, nearly all the coefficients $a(h)$ could be zero so studying an individual term is an inherent part of the problem.  One can obtain derivatives of $\zeta$ by differentiating the formulas with respect to the shift parameters. These shifts also allow for a structural viewpoint of the main terms.

The main term is written in terms of shifted products of the Riemann zeta function as well as some finite Euler products.  Let
\begin{equation}
\label{eq:Adef}
A_{\alpha,\beta,\gamma,\delta}(s) = \frac{\zeta(1+s + \alpha + \gamma)\zeta(1+s + \alpha + \delta)\zeta(1+s + \beta + \gamma)\zeta(1+s + \beta + \delta)}{\zeta(2+ 2s + \alpha + \beta + \gamma + \delta)}
\end{equation}
and set
\begin{equation}
\sigma_{\alpha,\beta}(n) = \sum_{n_1 n_2 = n} n_1^{-\alpha} n_2^{-\beta} = n^{-\alpha - \beta} \sum_{n_1 n_2 = n} n_1^{\alpha} n_2^{\beta}.
\end{equation}
Note $\sigma_{\alpha,\beta}(n) = n^{-\alpha} \sigma_{\alpha - \beta}(n)$, where $\sigma_{\lambda}(n) = \sum_{d | n} d^{\lambda}$.
Suppose $(h,k) = 1$, $p^{h_p} || h$ and $p^{k_p} || k$, and define
\begin{multline}
\label{eq:B}
B_{\alpha,\beta,\gamma,\delta,h,k}(s) = \prod_{p | h} \left(\frac{\sum_{j=0}^{\infty} \sigma_{\alpha,\beta}(p^j) \sigma_{\gamma,\delta}(p^{j + h_p}) p^{-j(s+1)}}{\sum_{j=0}^{\infty} \sigma_{\alpha,\beta}(p^j) \sigma_{\gamma,\delta}(p^{j}) p^{-j(s+1)}} 
\right)
\\
\times \prod_{p | k} 
\left(
\frac{\sum_{j=0}^{\infty} \sigma_{\alpha,\beta}(p^{j + k_p}) \sigma_{\gamma,\delta}(p^{j}) p^{-j(s+1)}}{\sum_{j=0}^{\infty} \sigma_{\alpha,\beta}(p^j) \sigma_{\gamma,\delta}(p^{j}) p^{-j(s+1)} }
\right).
\end{multline}
Let
\begin{equation}
Z_{\alpha,\beta,\gamma,\delta,h,k}(s) =  A_{\alpha,\beta,\gamma,\delta}(s) B_{\alpha,\beta,\gamma,\delta,h,k}(s).
\end{equation}
\begin{mytheo}
\label{thm:mainresult}
Let 
\begin{equation}
I(h,k) = \int_{-\infty}^{\infty} 
\left(\frac{h}{k}\right)^{-it}
\zeta({\textstyle \half + \alpha + it}) 
\zeta({\textstyle \half + \beta + it}) 
\zeta({\textstyle \half + \gamma - it}) 
\zeta({\textstyle \half + \delta - it}) w(t)dt,
\end{equation}
where $w(t)$ is a smooth, nonnegative function with support contained in $[\frac{T}{2}, 4T]$, satisfying $w^{(j)}(t) \ll_j T_0^{-j}$ for all $j=0,1,2,\dots$, where $T^{\half + \varepsilon} \ll T_0 \ll T$.  Suppose $(h,k) =1$, $hk \leq T^{\frac{2}{11} - \varepsilon}$, 
and that $\alpha, \beta, \gamma, \delta$ are complex numbers $\ll (\log{T})^{-1}$.  Then
\begin{multline}
\label{eq:mainresult}
I(h,k) = \frac{1}{\sqrt{hk}} \int_{-\infty}^{\infty} w(t) \left( Z_{\alpha,\beta,\gamma,\delta,h,k}(0)  
+ \left(\frac{t}{2 \pi}\right)^{-\alpha -\beta - \gamma - \delta}Z_{-\gamma, - \delta, -\alpha, -\beta,h,k}(0) \right. \\
+ \left(\frac{t}{2 \pi}\right)^{-\alpha - \gamma} Z_{-\gamma,\beta,-\alpha,\delta,h,k}(0) 
+ \left(\frac{t}{2 \pi}\right)^{-\alpha - \delta} Z_{-\delta,\beta,\gamma,-\alpha,h,k}(0) \\
\left.
+ \left(\frac{t}{2 \pi}\right)^{-\beta - \gamma} Z_{\alpha,-\gamma, -\beta,\delta,h,k}(0) 
+ \left(\frac{t}{2 \pi}\right)^{-\beta - \delta} Z_{\alpha,-\delta,\gamma,-\beta,h,k}(0)  
\right) dt \\
+ O(T^{\frac34 + \varepsilon} (hk)^{\frac78} (T/T_0)^{\frac94}).
\end{multline}
\end{mytheo}
Remarks.
\begin{itemize}
\item The general problem of estimating \eqref{eq:mollfiedsharpcutoff} can be deduced from Theorem \ref{thm:mainresult} simply by summing over $h$ and $k$, and taking $w$ to be a smooth approximation to the characteristic function $\chi_{[T/2,T]}$ of the interval $[T/2,T]$, vanishing $O(T^{1-\varepsilon})$ away from the endpoints (actually we take two different such functions, one bounded from above by $\chi_{[T/2,T]}$, and one bounded from below).
\item Here the size of the error term is entirely dependent upon Theorem 1 of \cite{DFI}.  Any improvement upon their result (using extra averaging or perhaps by using spectral methods, e.g.) would immediately improve our Theorem \ref{thm:mainresult}.  Our focus has been on development of the main terms; we have made no attempt to optimize the error terms.
\item In case $T_0 = T^{1-\varepsilon}$ then the main term (of size $\approx (hk)^{-\half} T$) is larger than the error term provided $hk \leq T^{\frac{2}{11}-\varepsilon}$.  The estimate of \eqref{eq:mainresult} continues to hold for any coprime $h$ and $k$ (so e.g. we have $I(h,k) = O(T^{1-\varepsilon})$ provided $hk \leq T^{\frac27 - \varepsilon}$).
\item The case $h=k=1$ with $T_0 = T^{12/13 + \varepsilon}$ can give an asymptotic formula for the fourth moment of the zeta function with an error of size $O(T^{12/13 + \varepsilon})$.  Here the error term depends on the exponent $9/4$ of the $(T/T_0)^{9/4}$ term in the error term above; \cite{DFI} remark that this exponent can likely be reduced (see their remark following Theorem 1).  This error term is not strong; we simply mention it to illustrate the flexibility of the result.
\item It is not obvious from inspection that the main term of \eqref{eq:mainresult} is holomorphic in terms of the shift parameters (for example, $Z_{\alpha,\beta,\gamma,\delta}(0)$ has poles at $\alpha=-\gamma$, $\alpha=-\delta$, $\beta = -\gamma$, and $\beta=-\delta$), but the symmetries of the expression imply that the poles cancel to form a holomorphic function.  Lemma 2.5.1 of \cite{CFKRS05} exhibits an integral representation for the permutation sum that proves the holomorphy.
\end{itemize}

\subsection{Structure of the proof}
The starting point of the proof is to use an approximate functional equation to express $I(h,k)$ as a divisor sum (see \eqref{eq:AFEforI}).  The divisor sum splits naturally into diagonal terms and off-diagonal terms.  The diagonal terms are (easily) treated in Section \ref{section:diagonal}.  To treat the off-diagonal terms we use the results of \cite{DFI}, which are reproduced in Section \ref{section:delta}.  In order to apply their results, we need to first simplify our formulas so that our test functions satisfy the conditions of their Theorem; we perform these manipulations in Section \ref{section:cleaning}.  The estimations up to this point determine the size of the error terms in Theorem \ref{thm:mainresult} (and hence the range of uniformity of $h$ and $k$ with respect to $T$.

The rest of this paper is for the purpose of simplifying the main terms given by Proposition \ref{prop:deltaconclusion}.  It turns out that there is a series of rather surprising identities which allows for considerable simplification of the main terms.  The delta method gives a main term that involves an arithmetical factor that is expressed as a certain sum of Ramanujan sums.  It is remarkable that this arithmetical factor satisfies a functional equation relating $s$ and $-s$; this formula is given by Theorem \ref{thm:CFE}.  This functional equation plays a key role in simplifying the main terms.

\subsection{Conventions}
We use the common practice in analytic number theory to let $\varepsilon$ denote an arbitrarily small positive constant which may vary from line to line.  We also assume that $T$ is sufficiently large with respect to $\varepsilon$ (so that we may say that $\zeta(1 + \alpha + \gamma + 2s)$ is holomorphic for $\text{Re}(s) > \varepsilon$, for example).
Furthermore, in our notation we occasionally drop the dependence of various quantities on the shift parameters.

\subsection{Acknowledgements}
The authors are indebted to Brian Conrey for his encouragement and support.  Some of this work was completed while the second author visited Bristol University, and he thanks them for the invitation and hospitality.

\section{Setup}
\subsection{The approximate functional equation}
We require an approximate functional equation for the product of zeta functions, motivated by a version used by \cite{HB79}.  Recall that the functional equation for the Riemann zeta function is given in its symmetric form by
\begin{equation}
\Lambda(s):=\pi^{-\frac{s}{2}} \Gamma({\textstyle \frac{s}{2}}) \zeta(s) = \Lambda(1-s).
\end{equation}
Thus
\begin{equation}
\zeta({\textstyle \half + s} ) = X(s) \zeta({\textstyle \half - s}),  
\end{equation}
where
\begin{equation}
X( s) = \pi^{s} \frac{\Gamma(\frac{\half -s}{2})}{\Gamma(\frac{\half + s}{2})}.
\end{equation}

We have
\begin{myprop}[Approximate functional equation]
\label{prop:AFE}
Let $G(s)$ be an even, entire function of rapid decay as 
$|s| \rightarrow \infty$ in any fixed strip $|\text{Re}(s)| \leq C$ and let
\begin{equation}
\label{eq:V}
V_{\alpha, \beta, \gamma, \delta,t}(x) = \frac{1}{2 \pi i} \int_{(1)} \frac{G(s)}{s} g_{\alpha, \beta, \gamma, \delta}(s,t) x^{-s} ds,
\end{equation}
where
\begin{equation}
g_{\alpha, \beta, \gamma, \delta}(s,t) = 
\frac{\Gamma\left(\frac{\half + \alpha + s +it}{2} \right)}{\Gamma\left(\frac{\half + \alpha +it }{2} \right)} 
\frac{\Gamma\left(\frac{\half + \beta + s +it}{2} \right)}{\Gamma\left(\frac{\half + \beta +it }{2} \right)} 
\frac{\Gamma\left(\frac{\half + \gamma + s -it}{2} \right)}{\Gamma\left(\frac{\half + \gamma -it }{2} \right)}
\frac{\Gamma\left(\frac{\half + \delta + s -it }{2} \right)}{\Gamma\left(\frac{\half + \delta -it }{2} \right)}.
\end{equation}
Furthermore, set
\begin{equation}
X_{\alpha,\beta,\gamma,\delta,t} = \pi^{\alpha + \beta + \gamma + \delta} 
\frac{\Gamma(\frac{\half -\alpha - it}{2})}{\Gamma(\frac{\half + \alpha + it}{2})}
\frac{\Gamma(\frac{\half -\beta - it}{2})}{\Gamma(\frac{\half + \beta + it}{2})}
\frac{\Gamma(\frac{\half -\gamma + it}{2})}{\Gamma(\frac{\half + \gamma - it}{2})}
\frac{\Gamma(\frac{\half -\delta + it}{2})}{\Gamma(\frac{\half + \delta - it}{2})}
.
\end{equation}
Then
\begin{multline}
\zeta({\textstyle \half + \alpha + it}) 
\zeta({\textstyle \half + \beta + it}) 
\zeta({\textstyle \half + \gamma - it}) 
\zeta({\textstyle \half + \delta - it})
\\
=
\sum_{m,n} \frac{\sigma_{\alpha,\beta}(m) \sigma_{\gamma,\delta}(n)}{(mn)^{\half }} \left(\frac{m}{n}\right)^{-it} V_{\alpha, \beta, \gamma, \delta,t} \left(\pi^2 mn \right)
\\
+ X_{\alpha,\beta,\gamma,\delta,t}
   \sum_{m,n} \frac{\sigma_{-\gamma,-\delta}(m) \sigma_{-\alpha,-\beta}(n)}{(mn)^{\half }} \left(\frac{m}{n}\right)^{-it} V_{ -\gamma, -\delta, -\alpha, -\beta, t} \left(\pi^2 mn\right) + O((1 + |t|)^{-2007}).
\end{multline}
\end{myprop}

Remark.  It will simplify later computations to proscribe certain zeros of $G(s)$ near $s = \half$; it is not essential.  Precisely, we assume $G$ is divisible by an even polynomial $Q_{\alpha,\beta,\gamma,\delta}(s)$ which is symmetric in the parameters $\alpha,\beta,\gamma,\delta$, invariant under $\alpha \rightarrow -\alpha$, or $\beta \rightarrow -\beta$, etc. and zero at $s = \half - \frac{\alpha + \gamma}{2}$ and $s = \half - \alpha$ (as well as other nearby points by symmetry), and that $G(s)/Q_{\alpha,\beta,\gamma,\delta}(s)$ is independent of $\alpha,\beta,\gamma,\delta$.  An admissible choice for $G(s)$ is $Q_{\alpha,\beta,\gamma, \delta}(s) \exp(s^2)$, but it is not necessary to specify a particular function $G$.

\begin{proof}
Let
\begin{equation}
\Lambda_{\alpha, \beta, \gamma, \delta,t}(s) = \Lambda(\thalf + s + \alpha +it) \Lambda(\thalf + s + \beta +it) \Lambda(\thalf + s + \gamma - it)\Lambda(\thalf + s + \delta - it) 
\end{equation}
and consider
\begin{equation}
I_1 = \frac{1}{2 \pi i} \int_{(1)} \Lambda_{\alpha, \beta, \gamma, \delta,t}(s) \frac{G(s)}{s} ds.
\end{equation}
Move the line of integration to $(-1)$, passing the pole at $s=0$ and the various poles of the zeta functions at $\half - \alpha - it$, etc.  Let $I_2$ be the new integral.  The residue at $s=0$ is
\begin{equation}
\Lambda(\thalf + \alpha +it) \Lambda(\thalf + \beta +it) \Lambda(\thalf +  \gamma - it)\Lambda(\thalf  + \delta - it).
\end{equation}
The other residues give $O((1 + |t|)^{-2007})$ due to the rapid decay of $G$ in the imaginary direction.

After the change of variables $s \rightarrow -s$ and the application of the functional equation
\begin{equation}
\Lambda_{\alpha, \beta, \gamma, \delta,t}(-s) = \Lambda_{-\gamma, -\delta, -\alpha, -\beta, t}(s),
\end{equation}
we obtain
\begin{equation}
I_2 = - \frac{1}{2 \pi i} \int_{(1)} \Lambda_{-\gamma, -\delta, -\alpha, -\beta,t}(s) \frac{G(s)}{s} ds.
\end{equation}
Set 
\begin{equation}
\zeta_{\alpha, \beta, \gamma, \delta,t}(s) = \zeta({\textstyle \half + \alpha +s + it}) 
\zeta({\textstyle \half + \beta +s + it}) 
\zeta({\textstyle \half + \gamma +s- it}) 
\zeta({\textstyle \half + \delta +s- it}),
\end{equation} 
and let
\begin{equation}
\Lambda_{\alpha, \beta, \gamma, \delta,t}(s) = \Gamma_{\alpha, \beta, \gamma, \delta,t}(s) \zeta_{\alpha, \beta, \gamma, \delta,t}(s),
\end{equation}
so that
\begin{equation}
\Gamma_{\alpha, \beta, \gamma, \delta,t}(s) = \pi^{-1 - 2s - \frac{\alpha + \beta + \gamma + \delta}{2}} 
\Gamma\left({\textstyle \frac{\half + \alpha + s + it}{2}} \right) \Gamma\left({\textstyle \frac{\half + \beta + s + it}{2}} \right)
\Gamma\left({\textstyle \frac{\half + \gamma + s - it}{2} } \right) \Gamma\left({\textstyle \frac{\half + \delta + s - it}{2}} \right).
\end{equation}
Then we have
\begin{multline}
\zeta_{\alpha, \beta, \gamma, \delta,t}(0) = \frac{1}{2 \pi i} \int_{(1)} \zeta_{\alpha, \beta, \gamma, \delta,t}(s) \frac{\Gamma_{\alpha, \beta, \gamma,\delta,t}(s)}{\Gamma_{\alpha, \beta, \gamma,\delta,t}(0)} \frac{G(s)}{s} ds 
\\
+ \frac{1}{2 \pi i} \int_{(1)} \zeta_{-\gamma, -\delta, -\alpha, -\beta, t}(s) \frac{\Gamma_{-\gamma, -\delta, -\alpha, -\beta, t}(s)}{\Gamma_{\alpha, \beta, \gamma,\delta, t}(0)} \frac{G(s)}{s} ds.
\end{multline}
An easy computation shows
\begin{equation}
\frac{\Gamma_{\alpha, \beta, \gamma,\delta,t}(s)}{\Gamma_{\alpha, \beta, \gamma,\delta,t}(0)} = \pi^{-2s} g_{\alpha, \beta, \gamma, \delta}(s,t)
\end{equation}
and
\begin{equation}
\frac{\Gamma_{-\gamma, -\delta, -\alpha, -\beta, t}(s)}{\Gamma_{\alpha, \beta, \gamma,\delta,t}(0)} 
= \pi^{-2s} X_{\alpha,\beta,\gamma,\delta,t} \, g_{-\gamma, -\delta, -\alpha, -\beta}(s, t),
\end{equation}
since
\begin{equation}
\frac{\Gamma_{-\gamma, -\delta, -\alpha, -\beta, t}(0)}{\Gamma_{\alpha, \beta, \gamma,\delta,t}(0)}
=
X_{\alpha,\beta,\gamma,\delta,t}.
\end{equation}
Expanding $\zeta_{*, *, *, *,t}(s)$ into absolutely convergent Dirichlet series and reversing the order of summation and integration completes the proof.
\end{proof}

\subsection{A formula for the twisted fourth moment}
We require an expression for the twisted integral $I(h,k)$.  Applying the approximate functional equation gives
\begin{multline}
\label{eq:AFEforI}
I(h,k) = \sum_{m,n} \frac{\sigma_{\alpha,\beta}(m) \sigma_{\gamma,\delta}(n)}{(mn)^{\half }}  \int_{-\infty}^{\infty} \left(\frac{hm}{kn}\right)^{-it} V_{\alpha, \beta, \gamma, \delta,t} \left(\pi^2 mn \right) w(t) dt
\\
+    \sum_{m,n} \frac{\sigma_{-\gamma,-\delta}(m) \sigma_{-\alpha,-\beta}(n) }{(mn)^{\half }} 
\int_{-\infty}^{\infty} \left(\frac{hm}{kn}\right)^{-it} X_{\alpha,\beta,\gamma,\delta,t} V_{-\gamma, -\delta, -\alpha, -\beta, t} \left(\pi^2 mn\right) w(t) dt.
\end{multline}
Let $I^{(1)}(h,k)$ be the first sum and $I^{(2)}(h,k)$ be the second sum above.  Opening the integral formula for $V$ gives
\begin{equation}
I^{(1)}(h,k) = \sum_{m,n} \frac{\sigma_{\alpha,\beta}(m) \sigma_{\gamma,\delta}(n)}{(mn)^{\half }}  \frac{1}{2 \pi i}  \int_{(1)} \frac{G(s)}{s} (\pi^2 mn)^{-s} \int_{-\infty}^{\infty} \left(\frac{hm}{kn}\right)^{-it}   g_{\alpha, \beta, \gamma, \delta}(s,t)   w(t) dt ds,
\end{equation}
and similarly
\begin{multline}
I^{(2)}(h,k) = \sum_{m,n} \frac{\sigma_{-\gamma,-\delta}(m) \sigma_{-\alpha,-\beta}(n) }{(mn)^{\half }}  
\\
\frac{1}{2 \pi i}  \int_{(1)} \frac{G(s)}{s} (\pi^2 mn)^{-s} \int_{-\infty}^{\infty} \left(\frac{hm}{kn}\right)^{-it}  X_{\alpha,\beta,\gamma,\delta,t} g_{-\gamma, -\delta, -\alpha, -\beta, }(s,t)   w(t) dt ds.
\end{multline}
An exercise with Stirling's approximation gives
\begin{equation}
X_{\alpha,\beta,\gamma,\delta,t} = \left(\frac{t}{2 \pi}\right)^{-\alpha - \beta - \gamma - \delta}  (1 + O(t^{-1}))
\end{equation}
and
\begin{equation}
\label{eq:gapprox}
g_{\alpha, \beta, \gamma, \delta}(s,t) = \left(\frac{t}{2}\right)^{2s} (1 + O(t^{-1})).
\end{equation}
for $t \rightarrow + \infty$.  In the approximation for $g(s,t)$ the dependence on $s$ in the error term is at most polynomial in $s$.  Note that to leading order $g(s,t)$ does not depend on the shift parameters.
The formulas for $I^{(1)}$ and $I^{(2)}$ are similar enough that we can study $I^{(1)}$ and then deduce an analogous formula for $I^{(2)}$ by changing the shift parameters via $\alpha \leftrightarrow -\gamma$, $\beta \leftrightarrow -\delta$, and multiplying by $(t/2\pi)^{-\alpha - \beta - \gamma - \delta}$.

\section{Diagonal terms}
\label{section:diagonal}
It is easy to see the contribution to $I^{(1)}$ of the diagonal terms $I_D^{(1)}(h,k)$ with $hm=kn$ is given by
\begin{equation}
I_D^{(1)}(h,k) = \frac{1}{\sqrt{hk}}  \int_{-\infty}^{\infty} w(t) 
\frac{1}{2 \pi i}  \int_{(1)} \frac{G(s)}{s} (\pi^2 hk)^{-s}  g_{\alpha, \beta, \gamma, \delta}(s,t) 
\sum_{n = 1}^{\infty} \frac{\sigma_{\alpha,\beta}(kn) \sigma_{\gamma,\delta}(hn)}{n^{1 + 2s}} 
ds dt.
\end{equation}
We compute the Dirichlet series
\begin{equation}
\sum_{n=1}^{\infty} \frac{\sigma_{\alpha,\beta}(kn) \sigma_{\gamma,\delta}(hn)}{n^{1+s}} 
\end{equation}
using that $\sigma_{\alpha,\beta}(kn) \sigma_{\gamma,\delta}(hn)$ is multiplicative in $n$.  In terms of its Euler product, it is
\begin{multline}
\label{eq:Eulerproduct}
\left(\prod_{(p, hk) = 1} \sum_{j=0}^{\infty} \frac{\sigma_{\alpha,\beta}(p^{j}) \sigma_{\gamma,\delta}(p^{j})}{p^{j(1+s)}}\right) \left( \prod_{p | h} \sum_{j=0}^{\infty} \frac{\sigma_{\alpha,\beta}(p^{j}) \sigma_{\gamma,\delta}(p^{h_p + j})}{p^{j(1+s)}} \right) 
\\
\times \left( \prod_{p | k} \sum_{j=0}^{\infty} \frac{\sigma_{\alpha,\beta}(p^{k_p + j}) \sigma_{\gamma,\delta}(p^{j})}{p^{j(1+s)}} \right),
\end{multline}
which equals $Z_{\alpha,\beta,\gamma,\delta,h,k}(s)$ since
\begin{equation}
\label{eq:fourzetas}
\sum_{n=1}^{\infty} \frac{\sigma_{\alpha,\beta}(n) \sigma_{\gamma,\delta}(n)}{n^{1+s}} = A_{\alpha,\beta,\gamma,\delta}(s).
\end{equation}
Hence
\begin{equation}
I_D^{(1)}(h,k) = \frac{1}{\sqrt{hk}}  \int_{-\infty}^{\infty} w(t) 
\frac{1}{2 \pi i}  \int_{(\varepsilon)} \frac{G(s)}{s} (\pi^2 hk)^{-s}  g_{\alpha, \beta, \gamma, \delta}(s,t) 
Z_{\alpha,\beta,\gamma,\delta,h,k}(2s) 
ds dt.
\end{equation}
Applying Stirling's approximation \eqref{eq:gapprox} gives
\begin{equation}
I_D^{(1)}(h,k) = \frac{1}{\sqrt{hk}}  \int_{-\infty}^{\infty} w(t) 
\frac{1}{2 \pi i}  \int_{(\varepsilon)} \frac{G(s)}{s} \left(\frac{t^2}{4\pi^2 hk}\right)^{s}   
Z_{\alpha,\beta,\gamma,\delta,h,k}(2s) 
ds dt + O((hk)^{-\half } T^{\varepsilon}).
\end{equation}
Note that $I_D^{(1)}(h,k)$ is holomorphic in $\alpha,\beta,\gamma, \delta$ sufficiently small.

Now move $\text{Re}(s)$ to $-\frac14 + \varepsilon$, crossing a pole at $s = 0$ as well as four poles at $2s = -\alpha - \gamma$, etc.  The integral on the new line is
\begin{equation}
\ll (hk)^{-\frac14 + \varepsilon} T^{\varepsilon} \int_{T/2}^{4T} t^{-\half + \varepsilon} |w(t)| dt \ll (hk)^{-\frac14 + \varepsilon} T^{\half + \varepsilon},
\end{equation}
uniformly in terms of the shift parameters.  The pole at $s=0$ gives
\begin{equation}
\frac{1}{\sqrt{hk}}  \int_{-\infty}^{\infty} 
Z_{\alpha,\beta,\gamma,\delta,h,k}(0) w(t) dt.
\end{equation}
The pole at $2s = -\alpha - \gamma$ gives
\begin{equation}
\frac{\text{Res}_{s=\frac{-\alpha-\gamma}{2}}(Z_{\alpha,\beta,\gamma,\delta,h,k}(2s)) }{(hk)^{\half - \frac{\alpha + \gamma}{2}}} \frac{G(\frac{-\alpha - \gamma}{2})}{\frac{-\alpha - \gamma}{2}}  
\int_{-\infty}^{\infty} \left(\frac{t}{2\pi}\right)^{-\alpha - \gamma} w(t) dt.
\end{equation}
The other three poles are gotten by taking obvious permutations of the shift variables.

The computation of the diagonal terms $I_{D}^{(2)}(h,k)$ of $I^{(2)}(h,k)$ is similar.  We get
\begin{multline}
I_{D}^{(2)}(h,k) = \frac{1}{\sqrt{hk}}  \int_{-\infty}^{\infty} \left(\frac{t}{2 \pi}\right)^{-\alpha - \beta - \gamma - \delta} w(t) 
\frac{1}{2 \pi i}  \int_{(\varepsilon)} \frac{G(s)}{s} \left(\frac{t^2}{4\pi^2 hk}\right)^{s}   
Z_{-\gamma,-\delta,-\alpha,-\beta,h,k}(2s) ds dt \\ + O((hk)^{-\half } T^{\varepsilon}).
\end{multline}
The pole at $s=0$ gives
\begin{equation}
\frac{1}{\sqrt{hk}}  \int_{-\infty}^{\infty} \left(\frac{t}{2\pi}\right)^{-\alpha - \beta -\gamma - \delta} w(t) 
Z_{-\gamma,-\delta,-\alpha,-\beta,h,k}(0) dt.
\end{equation}
The pole at $2s = \beta + \delta$ gives
\begin{equation}
\frac{\text{Res}_{s=\frac{\beta + \delta}{2}}(Z_{-\gamma,-\delta,-\alpha,-\beta,h,k}(2s))}{(hk)^{\half + \frac{\beta + \delta}{2}}}  \frac{G(\frac{\beta + \delta}{2})}{\frac{\beta + \delta}{2}} \int_{-\infty}^{\infty} \left(\frac{t}{2\pi}\right)^{-\alpha  -\gamma } w(t) dt.
\end{equation}
The other three poles are obtained by permuting the shifts.

These developments are summarized with the following
\begin{myprop}
\label{prop:diagonal}
We have
\begin{multline}
I_D^{(1)}(h,k) = \frac{1}{\sqrt{hk}} \int_{-\infty}^{\infty} Z_{\alpha,\beta,\gamma,\delta,h,k}(0) w(t) dt 
\\
 + J^{(1)}_{\alpha,\beta,\gamma,\delta} + J^{(1)}_{\beta, \alpha,\gamma,\delta} + J^{(1)}_{\alpha,\beta,\delta, \gamma} + J^{(1)}_{\beta, \alpha,\delta, \gamma} 
+ O\left(\frac{T^{\half + \varepsilon}}{(hk)^{\frac14}}  \right),
\end{multline}
and
\begin{multline}
I_{D}^{(2)}(h,k) = \frac{1}{\sqrt{hk}} \int_{-\infty}^{\infty} \left(\frac{t}{2\pi}\right)^{-\alpha - \beta - \gamma - \delta} Z_{-\gamma,-\delta,-\alpha,-\beta,h,k}(0) w(t) dt 
\\
 \qquad + J^{(2)}_{\alpha,\beta,\gamma,\delta} + J^{(2)}_{\beta, \alpha,\gamma,\delta} + J^{(2)}_{\alpha,\beta,\delta, \gamma} + J^{(2)}_{\beta, \alpha,\delta, \gamma} 
+ O\left(\frac{T^{\half + \varepsilon}}{(hk)^{\frac14}}  \right) ,
\end{multline}
where 
\begin{equation}
J^{(1)}_{\alpha,\beta,\gamma,\delta}  = \frac{\text{Res}_{s=\frac{-\alpha-\gamma}{2}}(Z_{\alpha,\beta,\gamma,\delta,h,k}(2s))}{(hk)^{\half - \frac{\alpha + \gamma}{2}}} \frac{G(\frac{-\alpha - \gamma}{2})}{\frac{-\alpha - \gamma}{2}}  
\int_{-\infty}^{\infty} \left(\frac{t}{2\pi}\right)^{-\alpha - \gamma} w(t) dt,
\end{equation}
and
\begin{equation}
J^{(2)}_{\alpha,\beta,\gamma,\delta}  = \frac{\text{Res}_{s=\frac{\beta + \delta}{2}}(Z_{-\gamma,-\delta,-\alpha,-\beta,h,k}(2s))}{(hk)^{\half + \frac{\beta + \delta}{2}}}  \frac{G(\frac{\beta + \delta}{2})}{\frac{\beta + \delta}{2}} \int_{-\infty}^{\infty} \left(\frac{t}{2\pi}\right)^{-\alpha  -\gamma } w(t) dt.
\end{equation}
\end{myprop}
Remarks.  
\begin{itemize} \item Since $G$ can be chosen from a wide class of functions, the terms of the form $J^{(1)}$ and $J^{(2)}$ should not contribute to $I(h,k)$.  Indeed, we show that the off-diagonal terms consist of other main terms minus the sum of $J^{(1)}$ and $J^{(2)}$'s, and thus these terms do not persist in the final formula for $I(h,k)$, as expected.
\item Note that $J^{(1)}_{\alpha,\beta,\gamma,\delta}$ (and $J^{(2)}$ of course) are not holomorphic in terms of the shift parameters because there are various poles, but of course the poles must cancel when the $J$'s and the term involving $Z$ are summed. 
\end{itemize}

\section{Off-diagonal terms: initial cleaning}
\label{section:cleaning}
Now we begin to treat the contribution to $I^{(1)}(h,k)$ of the off-diagonal terms $I_O^{(1)}(h,k)$, say (similarly, let $I_{O}^{(2)}(h,k)$ denote the off-diagonal term contribution to $I^{(2)}(h,k)$). 
First note that we may truncate the sum over $m$ and $n$ so that $mn \leq T^{2 + \varepsilon}$ by moving $\text{Re}(s)$ to the right, using \eqref{eq:gapprox}.

The goal of the rest of this paper is to prove the following
\begin{myprop}
\label{prop:offdiagonal}
We have
\begin{multline}
I_{O}^{(1)}(h,k) + I_{O}^{(2)}(h,k) = \frac{1}{\sqrt{hk}} \int_{-\infty}^{\infty} w(t) \left( \left(\frac{t}{2 \pi}\right)^{-\alpha - \gamma} Z_{-\gamma,\beta,-\alpha,\delta,h,k}(0) \right. \\
 + \left(\frac{t}{2 \pi}\right)^{-\alpha - \delta} Z_{-\delta,\beta,\gamma,-\alpha,h,k}(0)  
 + \left(\frac{t}{2 \pi}\right)^{-\beta - \gamma} Z_{\alpha,-\gamma, -\beta,\delta,h,k}(0) \\
\left. + \left(\frac{t}{2 \pi}\right)^{-\beta - \delta} Z_{\alpha,-\delta,\gamma,-\beta,h,k}(0)  \right) dt 
- J^{(1)}_{\alpha,\beta,\gamma,\delta} - J^{(1)}_{\beta, \alpha,\gamma,\delta} - J^{(1)}_{\alpha,\beta,\delta, \gamma} - J^{(1)}_{\beta, \alpha,\delta, \gamma} 
\\
- J^{(2)}_{\alpha,\beta,\gamma,\delta} - J^{(2)}_{\beta, \alpha,\gamma,\delta} - J^{(2)}_{\alpha,\beta,\delta, \gamma} - J^{(2)}_{\beta, \alpha,\delta, \gamma}
+ O(T^{\frac34 + \varepsilon} (hk)^{\frac78} (T/T_0)^{\frac54}),
\end{multline}
uniformly for $\alpha,\beta,\gamma,\delta \ll (\log{T})^{-1}$.
\end{myprop}
We obtain Theorem \ref{thm:mainresult} by combining Propositions \ref{prop:diagonal} and \ref{prop:offdiagonal}.  Note that all these $J$ terms exactly cancel!

Let
\begin{equation}
f^*(x,y) =  \frac{1}{2 \pi i}  \int_{(\varepsilon)} \frac{G(s)}{s} \left(\frac{hk}{\pi^2 xy}\right)^{s} \frac{1}{T}  \int_{-\infty}^{\infty}  \left(\frac{x}{y}\right)^{-it}   g(s,t)   w(t) dt ds,
\end{equation}
so that
\begin{equation}
I_O^{(1)}(h,k) =  T \mathop{\sum_m \sum_n}_{hm \neq kn} \frac{\sigma_{\alpha,\beta}(m) \sigma_{\gamma,\delta}(n)}{\sqrt{mn}} f^*(hm,kn).
\end{equation}
Our goal is to apply the result of \cite{DFI} to this sum.

Now apply a dyadic partition of unity to the sums over $m$ and $n$.  That is, suppose $W_0(x)$ is a smooth, nonnegative function with support in $[1,2]$ such that
\begin{equation}
\sum_M W_0({\textstyle \frac{x}{M}}) = 1,
\end{equation}
where $M$ runs over a sequence of real numbers, with $\# \{M | M \leq X \} \ll \log{X}$.  Let
\begin{equation}
I_{M,N}(h,k) = \frac{T}{\sqrt{MN}} \mathop{\sum  \sum }_{hm \neq kn} \sigma_{\alpha,\beta}(m) \sigma_{\gamma,\delta}(n)
W\left({\textstyle \frac{m}{M} }\right) W\left({\textstyle \frac{n}{N} }\right) f^*(hm,kn),
\end{equation}
where
\begin{equation}
W(x) = x^{-\half} W_0(x).
\end{equation}
Note that we may assume $MN \leq T^{2 + \varepsilon}$.  Then
\begin{equation}
\sum_{M, N} I_{M,N}(h,k) = I_O^{(1)}(h,k).
\end{equation}

It is not difficult to see that $f^*(x,y)$ is small unless $x$ and $y$ are close to each other, due to cancellation in the integral arising from the factor $(x/y)^{-it}$.  Precisely, by repeated integration by parts with respect to $t$, we obtain
\begin{equation}
\frac{1}{T} \int_{-\infty}^{\infty} \left(\frac{x}{y}\right)^{-it}   g(s,t)   w(t) dt \ll_j
\frac{1}{T |\log(x/y)|^j} \int_{-\infty}^{\infty} \left| \frac{\partial^j}{\partial t^j} g(s,t) w(t) \right| dt
\ll_j \frac{P_j(|s|) T^{2 \text{Re}(s)}}{|\log(x/y)|^j T_0^{j}},
\end{equation}
for any $j=0,1,2, \dots$, where $P_j$ is a polynomial.  Thus we may assume  $|\log(\frac{x}{y})| \ll T_0^{-1 + \varepsilon}$ by taking $j$ sufficiently large.  

Letting $hm - kn = r$, we get
\begin{multline}
\label{eq:thermoscap}
I_{M,N}(h,k) 
\\
= \frac{T}{\sqrt{MN}} \sum_{r \neq 0} \mathop{\sum \sum}_{\substack{hm - kn = r \\ |\log(\frac{hm}{kn})| \ll T_0^{-1 + \varepsilon}}} \sigma_{\alpha,\beta}(m) \sigma_{\gamma,\delta}(n) W\left({\textstyle \frac{m}{M} }\right) W\left({\textstyle \frac{n}{N} }\right) f^*(hm,kn) + O(T^{-2007}).
\end{multline}
Note we may assume 
\begin{equation}
hM \asymp kN.
\end{equation}

If $x-y = r$ then
\begin{equation}
f^*(x,y) =\frac{1}{2 \pi i}  \int_{(\varepsilon)} \frac{G(s)}{s} \left(\frac{hk}{\pi^2 xy}\right)^{s} \frac{1}{T}  \int_{-\infty}^{\infty} \left(1 + \frac{r}{y}\right)^{-it}  g(s,t)   w(t) dt ds.
\end{equation}

We summarize this development with
\begin{myprop}
\label{prop:JO}
We have
\begin{multline}
\label{eq:JO}
I_{M,N}(h,k) = \frac{T}{\sqrt{MN}} \sum_{0 < |r| \ll \frac{\sqrt{hkMN}}{T_0} T^{\varepsilon}} \mathop{\sum \sum}_{hm - kn = r} \sigma_{\alpha,\beta}(m) \sigma_{\gamma,\delta}(n) f(hm,kn) 
+ O\left(T^{-2007} \right),
\end{multline}
where
\begin{equation}
\label{eq:f}
f(x,y) = W\left(\frac{x}{hM}\right) W\left(\frac{y}{kN}\right) \frac{1}{2 \pi i}  \int_{(\varepsilon)} \frac{G(s)}{s} \left(\frac{hk}{\pi^2 xy}\right)^{s} \frac{1}{T}  \int_{-\infty}^{\infty} \left(1 + \frac{r}{y}\right)^{-it} g(s,t)   w(t) dt ds.
\end{equation}
\end{myprop}

\section{Applying the delta method}
\label{section:delta}
The expression \eqref{eq:JO} is well-suited for application of the main result of \cite{DFI}, which we reproduce here for completeness.
\begin{mytheo}[Duke, Friedlander, Iwaniec]
Let $f$ be a smooth function on $\mr^+ \times \mr^+$ satisfying
\begin{equation}
x^i y^j f^{(i,j)}(x,y) \ll \left(1 + \frac{x}{X}\right)^{-1} \left(1 + \frac{y}{Y}\right)^{-1} P^{i + j},
\end{equation}
for some $P, X, Y \geq 1$ and all $i, j \geq 0$ with the implied constant depending on $i$ and $j$ alone.  Let
\begin{equation}
D_f(h,k;r) = \sum_{hm - kn = r} d(m) d(n) f(hm,kn).
\end{equation}
Then
\begin{equation}
D_f(h,k;r) = \int_{\max(0,r)}^{\infty} g(x, x - r) dx + O(P^{5/4} (X + Y)^{1/4} (XY)^{1/4 + \varepsilon}),
\end{equation}
where $g(x,y) = f(x,y) \Lambda_{hkr}(x,y)$ and $\Lambda$ is a certain explicitly given infinite series.  The implied constant  depends on $\varepsilon$ only.
\end{mytheo}
We chose not to explicitly write the main term because we require a modified expression due to the fact that we require sums involving $\sigma_{\alpha,\beta}(m) \sigma_{\gamma,\delta}(n)$ rather than $d(m) d(n)$.  Introducing these shift parameters slightly perturbs the main term but does not alter the error terms (since the shifts are small).  Following the arguments of \cite{DFI}, one deduces
\begin{multline}
\label{eq:deltashift}
\sum_{hm - kn = r} \sigma_{\alpha,\beta}(m) \sigma_{\gamma,\delta}(n) f(hm,kn) = N_{\alpha,\beta,\gamma,\delta}(h,k;r) 
+ N_{\beta,\alpha,\gamma,\delta}(h,k;r) 
\\
+ N_{\alpha,\beta,\delta,\gamma}(h,k;r) + N_{\beta,\alpha,\delta,\gamma}(h,k;r) + O(P^{5/4} (X + Y)^{1/4} (XY)^{1/4 + \varepsilon}),
\end{multline}
where
\begin{multline}
N_{\alpha,\beta,\gamma,\delta}(h,k;r) = \frac{\zeta(1- \alpha + \beta) \zeta(1 - \gamma + \delta)}{h^{1 - \alpha} k^{1 - \gamma}} 
\int x^{-\alpha} (x-r)^{-\gamma} f(x,x-r) dx
\\
\sum_{l=1}^{\infty} \frac{c_l(r) (h,l)^{1 - \alpha + \beta} (k,l)^{1 - \gamma + \delta}}{l^{2-\alpha + \beta - \gamma + \delta}}.
\end{multline}
Here
\begin{equation}
c_l(r) = \sumstar_{a \shortmod{l}} e\left(\frac{ar}{l}\right)
\end{equation}
is the Ramanujan sum.  To derive this expression for the main term we used a modified form of Jutila's version of the Voronoi summation formula, namely that for $\alpha \neq \beta$,
\begin{multline}
\sum_n \sigma_{\alpha,\beta}(n) e\left(\frac{dn}{q}\right) g(n) 
\\
= \frac{\zeta(1-\alpha+\beta)}{q^{1-\alpha+\beta}} \int x^{-\alpha} g(x) dx + \frac{\zeta(1+\alpha-\beta)}{q^{1+\alpha-\beta}} \int x^{-\beta} g(x) dx + \dots,
\end{multline}
which can be deduced from the functional equation of the Estermann function as modified by Motohashi; see \cite{Moto}, Lemma 3.7.

Note that $N_{\alpha,\beta,\gamma,\delta}(h,k;r)$ has a simple pole when $\alpha = \beta$ or $\gamma = \delta$, but the sum of the four such terms in \eqref{eq:deltashift} is holomorphic because the poles cancel by symmetry.  In the forthcoming development of the main terms we shall concentrate on the contribution of one of the four such terms and obtain the others by a simple symmetry argument.  A slight logical issue arises due to the lack of holomorphy of $N_{\alpha,\beta,\gamma,\delta}$ so that the uniformity of the error terms with respect to the shift parameters becomes compromised.  Fortunately, it is not difficult to see that the uniformity must be regained after summing the four terms, and we shall leave some of these simple arguments implicit in this work.

Now we apply this result to $f$ given by \eqref{eq:f}, which satisfies the necessary conditions with
\begin{equation}
X = hM, \qquad Y = kN, \qquad P = \frac{T}{T_0} T^{\varepsilon}.
\end{equation}
Note $X \asymp Y \asymp \sqrt{hkMN}$.
Hence we obtain, using Proposition \ref{prop:JO},
\begin{multline}
I_{M,N} = \frac{T}{\sqrt{MN}} \sum_{0 \neq |r| \ll \frac{T^{\varepsilon} \sqrt{hkMN}}{T_0} } \left( N_{\alpha,\beta,\gamma,\delta}(h,k;r) + \dots + N_{\beta,\alpha,\delta,\gamma}(h,k;r) \right)
\\
+ \sum_{0 \neq |r| \ll \frac{T^{\varepsilon} \sqrt{hkMN}}{T_0} } O\left(T^{\varepsilon} \left(T/T_0 \right)^{\frac54}(hkMN)^{\frac38} \right) .
\end{multline}
Summing over $r$ gives
\begin{multline}
I_{M,N} = \frac{T}{\sqrt{MN}} \sum_{0 \neq |r| \ll \frac{T^{\varepsilon} \sqrt{hkMN}}{T_0} } \left(N_{\alpha,\beta,\gamma,\delta}(h,k;r) + \dots + N_{\beta,\alpha,\delta,\gamma}(h,k;r) \right) 
\\
+
O\left(\frac{(hkMN)^{\frac78}}{T_0} \frac{T^{\frac54}}{T_0^{\frac54}} T^{\varepsilon}  \right).
\end{multline}

Let
\begin{equation}
\label{eq:I}
I_{\alpha,\beta,\gamma,\delta}^{(1)} = \sum_{M,N} \sum_{r \neq 0} \frac{T}{\sqrt{MN}}  N_{\alpha,\beta,\gamma,\delta}(h,k;r).
\end{equation}
If $|r| \geq T_0^{-1} \sqrt{hkMN} T^{\varepsilon}$ then the usual integration by parts argument shows that $N_{\alpha,\beta,\gamma,\delta}(h,k;r)$ is small, so we may freely extend the summation to all $r \neq 0$.
In summary, the delta method then gives the following
\begin{myprop}
\label{prop:deltaconclusion}
We have
\begin{equation}
I_O^{(1)}(h,k) = I_{\alpha,\beta,\gamma,\delta}^{(1)} + I_{\beta, \alpha, \gamma,\delta}^{(1)} + I_{\alpha,\beta ,\delta, \gamma}^{(1)} + I_{\beta, \alpha, \delta, \gamma}^{(1)} + 
O(T^{\frac34 + \varepsilon} (hk)^{\frac78} (T/T_0)^{\frac94}).
\end{equation}
\end{myprop}
Here the main terms are of rough size $(hk)^{-\half} T_0$, so if $T_0 \gg T^{1 - \varepsilon}$ then this expression has power savings in the error term provided $hk \ll T^{\frac{2}{11} - \varepsilon}$.

Now our main task is to develop simpler formulas for the main terms, which we carry out in the following section.

\section{Computing the main terms}
\label{section:mainterms}
\subsection{Integral manipulations}
In this section we compute $I_{\alpha,\beta,\gamma,\delta}^{(1)}$ which is
given by
\begin{multline}
I_{\alpha,\beta,\gamma,\delta}^{(1)} =  \sum_{r \neq 0} \sum_{M,N} \frac{T}{\sqrt{MN}} \frac{\zeta(1- \alpha + \beta) \zeta(1 - \gamma + \delta)}{h^{1 - \alpha} k^{1 - \gamma}} 
\int_{\max\{0,r\}}^{\infty} x^{ - \alpha} (x-r)^{-\gamma} 
\\
W\left(\frac{x}{hM}\right) W\left(\frac{x-r}{kN}\right)  \frac{1}{2 \pi i}  \int_{(\varepsilon)} \frac{G(s)}{s} \left(\frac{hk}{\pi^2 x(x-r)}\right)^{s} \frac{1}{T}  \int_{-\infty}^{\infty} \left(1-\frac{r}{x}\right)^{it} g(s,t)   w(t) dt ds dx
\\
\sum_{l=1}^{\infty} \frac{c_l(r) (h,l)^{1 - \alpha + \beta} (k,l)^{1 - \gamma + \delta}}{l^{2-\alpha + \beta - \gamma + \delta}}.
\end{multline}
We will show
\begin{multline}
I_{\alpha,\beta,\gamma,\delta}^{(1)} =  \frac{\zeta(1- \alpha + \beta) \zeta(1 - \gamma + \delta)}{\zeta(2-\alpha + \beta - \gamma + \delta)} \left[ \frac{1}{\sqrt{hk}} 
\int_{-\infty}^{\infty} w(t) \left(\frac{t}{2 \pi}\right)^{-\alpha - \gamma} \right.
\\
\frac{1}{2 \pi i}  \int_{(\varepsilon)} \frac{G(s)}{s}  
\zeta(1-\alpha  -\gamma -2s) \zeta(1 + \beta + \delta + 2s) h^{\alpha} k^{\gamma} (hk)^s C_{\alpha,\beta,\gamma,\delta,h,k}(s)
ds dt
\\
\left. + O\left(\frac{1}{\sqrt{hk}} (hkT)^{\varepsilon} \right) \right],
\end{multline}
where $C_{\alpha,\beta,\gamma,\delta}(s)$ is a certain finite Euler product given below by Corollary \ref{coro:arithmeticfactor}.

Using $W(x) = x^{-\half} W_0(x)$ and summing over $M$ and $N$ gives
\begin{multline}
I_{\alpha,\beta,\gamma,\delta}^{(1)} = \frac{\zeta(1- \alpha + \beta) \zeta(1 - \gamma + \delta)}{h^{\half - \alpha} k^{\half - \gamma}} \sum_{r \neq 0}
\sum_{l=1}^{\infty}  \frac{c_l(r) (h,l)^{1 - \alpha + \beta} (k,l)^{1 - \gamma + \delta}}{l^{2-\alpha + \beta - \gamma + \delta}}
\\
\int_{\max\{0,r\}}^{\infty} x^{ -\half- \alpha} (x-r)^{-\half -\gamma} 
\frac{1}{2 \pi i}  \int_{(\varepsilon)} \frac{G(s)}{s} \left(\frac{hk}{\pi^2 x(x-r)}\right)^{s}  \int_{-\infty}^{\infty} \left(1-\frac{r}{x}\right)^{it} g(s,t)   w(t) dt ds dx.
\end{multline}
Let $I = I^+ + I^-$ where $I^+$ corresponds to the sum over $r > 0$, and $I^-$ corresponds to the complement.  The computations of $I^+$ and $I^-$ are similar yet slightly different.  We have
\begin{multline}
I^+ = \frac{\zeta(1- \alpha + \beta) \zeta(1 - \gamma + \delta)}{h^{\half - \alpha} k^{\half - \gamma}} \sum_{r =1}^{\infty}
\sum_{l=1}^{\infty}  \frac{c_l(r) (h,l)^{1 - \alpha + \beta} (k,l)^{1 - \gamma + \delta}}{l^{2-\alpha + \beta - \gamma + \delta}}
\\
\int_{r}^{\infty} x^{ -\half- \alpha} (x-r)^{-\half -\gamma} 
\frac{1}{2 \pi i}  \int_{(\varepsilon)} \frac{G(s)}{s} \left(\frac{hk}{\pi^2 x(x-r)}\right)^{s}  \int_{-\infty}^{\infty} \left(1-\frac{r}{x}\right)^{it} g(s,t)   w(t) dt ds dx,
\end{multline}
and
\begin{multline}
I^- = \frac{\zeta(1- \alpha + \beta) \zeta(1 - \gamma + \delta)}{h^{\half - \alpha} k^{\half - \gamma}} \sum_{r =1}^{\infty}
\sum_{l=1}^{\infty}  \frac{c_l(r) (h,l)^{1 - \alpha + \beta} (k,l)^{1 - \gamma + \delta}}{l^{2-\alpha + \beta - \gamma + \delta}}
\\
\int_{0}^{\infty} x^{ -\half- \alpha} (x+r)^{-\half -\gamma} 
\frac{1}{2 \pi i}  \int_{(\varepsilon)} \frac{G(s)}{s} \left(\frac{hk}{\pi^2 x(x+r)}\right)^{s}  \int_{-\infty}^{\infty} \left(1+\frac{r}{x}\right)^{it} g(s,t)   w(t) dt ds dx.
\end{multline}

We first compute $I^+$. Let $K^{\pm}$ be the triple integral appearing in the above expressions for $I^{\pm}$.  Changing variables via $x \rightarrow rx + r$ gives
\begin{multline}
K^+ = r^{-\alpha - \gamma}
\int_{0}^{\infty} (x+1)^{ -\half- \alpha} x^{-\half -\gamma} 
\\
\frac{1}{2 \pi i}  \int_{(\varepsilon)} \frac{G(s)}{s} \left(\frac{hk}{\pi^2 r^2 x(x+1)}\right)^{s}  \int_{-\infty}^{\infty} x^{it}(1+x)^{-it} g(s,t)   w(t) dt ds dx.
\end{multline}
Similarly,
\begin{multline}
K^- = r^{-\alpha - \gamma}
\int_{0}^{\infty} (x+1)^{ -\half- \alpha} x^{-\half -\gamma} 
\\
\frac{1}{2 \pi i}  \int_{(\varepsilon)} \frac{G(s)}{s} \left(\frac{hk}{\pi^2 r^2 x(x+1)}\right)^{s}  \int_{-\infty}^{\infty} x^{-it}(1+x)^{it} g(s,t)   w(t) dt ds dx.
\end{multline}
Thus we obtain
\begin{multline}
I^{\pm} = \frac{\zeta(1- \alpha + \beta) \zeta(1 - \gamma + \delta)}{h^{\half - \alpha} k^{\half - \gamma}} 
\sum_{r =1}^{\infty}
\sum_{l=1}^{\infty}  \frac{c_l(r) (h,l)^{1 - \alpha + \beta} (k,l)^{1 - \gamma + \delta}}{l^{2-\alpha + \beta - \gamma + \delta} r^{\alpha + \gamma}}
\\
\int_0^{\infty} (x+1)^{ -\half- \alpha} x^{-\half -\gamma} 
\int_{-\infty}^{\infty} x^{\pm it} (1+x)^{\mp it}
\frac{1}{2 \pi i}  \int_{(\varepsilon)} \frac{G(s)}{s}  g(s,t)   \left(\frac{hk}{\pi^2 r^2 x(1+x)}\right)^{s} ds dt dx.
\end{multline}
Using 3.194.3 of \cite{GR} and the well-known representation of the beta function in terms of gamma functions, we have
\begin{align}
\int_0^{\infty} (1+x)^{ -\half- \alpha -s \mp it} x^{-\half -\gamma -s \pm it} dx &= B(\thalf - \alpha - s \pm it, \alpha + \gamma +2s) 
\\
& = \frac{\Gamma(\half - \alpha -s \pm it) \Gamma(\alpha + \gamma + 2s)}{\Gamma(\half + \gamma + s \pm it)}.
\end{align}
Thus rearranging the orders of integration gives
\begin{multline}
I^{\pm} = \frac{\zeta(1- \alpha + \beta) \zeta(1 - \gamma + \delta)}{h^{\half - \alpha} k^{\half - \gamma}} 
\sum_{r =1}^{\infty}
\sum_{l=1}^{\infty}  \frac{c_l(r) (h,l)^{1 - \alpha + \beta} (k,l)^{1 - \gamma + \delta}}{l^{2-\alpha + \beta - \gamma + \delta} r^{\alpha + \gamma}}
\\ 
\int_{-\infty}^{\infty} w(t)
\frac{1}{2 \pi i}   \int_{(\varepsilon)}  \frac{G(s)}{s}  g(s,t)   \left(\frac{hk}{\pi^2 r^2}\right)^{s} \frac{\Gamma(\half - \alpha -s \pm it) \Gamma(\alpha + \gamma + 2s)}{\Gamma(\half + \gamma + s \pm it)} ds dt.
\end{multline}
By Stirling's approximation,
\begin{equation}
\frac{\Gamma(\half - \alpha -s \pm it)}{\Gamma(\half + \gamma + s \pm it)} = t^{-\alpha-\gamma-2s} \exp\left(\frac{\pi i}{2} \text{sgn}(t) (-\alpha-\gamma-2s)\right) \left(1 + O\left(\frac{1 + |s|^2}{t}\right) \right).
\end{equation}

We conclude that
\begin{multline}
\label{eq:eart}
I^{(1)}_{\alpha,\beta,\gamma,\delta} = \frac{\zeta(1- \alpha + \beta) \zeta(1 - \gamma + \delta)}{h^{\half - \alpha} k^{\half - \gamma}} 
\sum_{r =1}^{\infty}
\sum_{l=1}^{\infty}  \frac{c_l(r) (h,l)^{1 - \alpha + \beta} (k,l)^{1 - \gamma + \delta}}{l^{2-\alpha + \beta - \gamma + \delta} r^{\alpha + \gamma}} \int_{-\infty}^{\infty} w(t)
\\
\frac{1}{2 \pi i}  \int_{(\varepsilon)} \frac{G(s)}{s}  g(s,t)   \left(\frac{hk}{\pi^2 r^2}\right)^{s} 
\Gamma(\alpha + \gamma + 2s) t^{-\alpha - \gamma - 2s} 2 \cos({\textstyle \frac{\pi }{2}}(\alpha + \gamma + 2s))\left(1 + O\left(\frac{1 + |s|^2}{t}\right)\right).
\end{multline}
At this point we shall work with the arithmetical sum over $l$ and $r$.  We first move the $s$-line of integration to $1$ so that the sums converge absolutely.

\subsection{An arithmetical sum}
Let
\begin{equation}
F(a,b,c) = \sum_{r = 1}^{\infty}
\sum_{l=1}^{\infty}  \frac{c_l(r) (h,l)^{a} (k,l)^{b}}{l^{a+b} r^{c}}.
\end{equation}
The desired sum in \eqref{eq:eart} is $F(1-\alpha + \beta, 1-\gamma + \delta, \alpha + \gamma + 2s)$.  We have
\begin{mylemma}  Suppose $\text{Re}(a+b) > 1$ and $\text{Re}(c) > 0$.  Then
\begin{multline}
F(a,b,c+1) = 
\frac{\zeta(1+c) \zeta(a+b+c)}{\zeta(a+b)} 
\\
\prod_{p^{h_p} || h} \left(\frac{(1 - p^{-b})(1- p^{-a-b-c}) + p^{-b}(1-p^{-a})(1-p^{-c})p^{h_p(-b-c)}}{(1 - p^{-b- c })(1 - p^{-a-b})} \right)
\\
\prod_{p^{k_p} || k} \left(\frac{(1 - p^{-a})(1- p^{-a-b-c}) + p^{-a}(1-p^{-b})(1-p^{-c})p^{k_p(-a-c)}}{(1 - p^{-a- c })(1 - p^{-a-b})} \right).
\end{multline}
\end{mylemma}
A brief computation gives
\begin{mycoro}
\label{coro:arithmeticfactor}
\begin{equation}
\sum_{r =1}^{\infty}
\sum_{l=1}^{\infty}  \frac{c_l(r) (h,l)^{1 - \alpha + \beta} (k,l)^{1 - \gamma + \delta}}{l^{2-\alpha + \beta - \gamma + \delta} r^{\alpha + \gamma + 2s}} = \frac{\zeta(\alpha + \gamma +2s) \zeta(1 + \beta + \delta + 2s)}{\zeta(2-\alpha + \beta - \gamma + \delta)} C_{\alpha,\beta,\gamma,\delta,h,k}(s),
\end{equation}
where
\begin{equation}
C_{\alpha,\beta,\gamma,\delta,h,k}(s) = C_{\alpha,\beta,\gamma,\delta,h}(s) C_{\gamma,\delta, \alpha,\beta,k}(s),
\end{equation}
and
\begin{equation}
\label{eq:Calt}
C_{\alpha,\beta,\gamma,\delta, h}(s) = \prod_{p | h} (1 - p^{-2 + \alpha - \beta + \gamma - \delta})^{-1} \prod_{p^{h_p} || h} 
\left(\frac{C^{(0)}(s) - p^{-1} C^{(1)}(s) + p^{-2} C^{(2)}(s)}{1 - p^{- \alpha  - \delta    -2s }} \right),
\end{equation}
and where
\begin{align}
C^{(0)}(s) & = 1 - p^{(-\alpha - \delta - 2s)(1 + h_p)}, \\
C^{(1)}(s) & = (p^{\gamma - \delta} + p^{\alpha - \beta} p^{-\alpha - \delta - 2s}) (1 - p^{(-\alpha - \delta - 2s)h_p}), \\
C^{(2)}(s) & = p^{\alpha - \beta + \gamma - \delta}(p^{-\alpha - \delta - 2s} - p^{(-\alpha - \delta - 2s)h_p}).
\end{align}
Here we write $C^{(i)}(s)$ as shorthand for $C_{\alpha,\beta,\gamma,\delta,h}^{(i)}(s)$.
\end{mycoro}

\begin{proof}
Using
\begin{equation}
c_l(r) = \sum_{d |(l,r)} d \mu(l/d)
\end{equation}
gives
\begin{align}
F(a,b,1+c) &= \sum_{r = 1}^{\infty}
\sum_{l=1}^{\infty}  \frac{ (h,l)^{a} (k,l)^{b}}{l^{a+b} r^{1+c}} \sum_{d |(l,r)} d \mu(l/d) 
\\
&= \zeta(1+c)
\sum_{l=1}^{\infty}  \frac{ (h,l)^{a} (k,l)^{b}}{l^{a+b+c}} \sum_{d |l} d^c \mu(d),
\end{align}
by applying the change of variables $r \rightarrow dr$, summing over $r$ and applying the change of variables $d \rightarrow l/d$.  Now we express $F$ in terms of its Euler product; we shall focus on each prime separately.  Precisely, we study the Euler product of $F(a,b,1+c)/\zeta(1+c)$.  Note that if $p \nmid hk$ then the local factor is
\begin{align}
1 + (1 -p^{ c}) \sum_{j=1}^{\infty} \frac{ 1}{(p^j)^{a+b+c}} 
& =
1 + (1 -p^{c}) \frac{p^{-a-b-c}}{1- p^{-a-b-c}}
\\
= \frac{1- p^{-a-b-c} + (1 -p^{ c})p^{-a-b-c}}{1- p^{-a-b-c}}
& = \frac{1- p^{-a-b}}{1- p^{-a-b-c}}.
\end{align}
Similarly, if $p^{h_p} || h$, then we have that the local factor of $F(a,b,1+c)/\zeta(1+c)$ is
\begin{equation}
1 + (1 -p^{ c}) \sum_{j=1}^{\infty} \frac{ (p^{h_p},p^j)^{a}}{(p^j)^{a+b+c}}.
\end{equation}
Note
\begin{align}
\sum_{j=1}^{\infty} \frac{ (p^{h_p},p^j)^{a}}{(p^j)^{a+b+c}} = \sum_{j=1}^{h_p} \frac{ p^{ja}}{(p^j)^{a+b+c}} + \sum_{j=h_p + 1}^{\infty} \frac{ p^{a h_p}}{(p^j)^{a+b+c}}
\\
= p^{-b-c} \frac{1 - p^{h_p(-b-c)}}{1 - p^{-b- c}} +  \frac{p^{-a -b-c } p^{h_p(-b-c)}}{1 - p^{-a-b-c }}
\\
= p^{-b-c} \frac{(1 - p^{h_p(-b-c )})(1 - p^{-a-b-c}) + p^{-a} p^{h_p(-b-c)}(1 - p^{-b- c})
}{(1 - p^{-b- c })(1 - p^{-a-b-c})}
\\
= p^{-b-c} \frac{1-p^{-a-b-c} - p^{h_p(-b-c)} + p^{-a}p^{h_p(-b-c)}}{(1 - p^{-b- c})(1 - p^{-a-b-c})}.
\end{align}
Thus the local factor is
\begin{align}
\frac{(1 - p^{-b- c })(1 - p^{-a-b-c }) +  p^{-b-c}(1-p^{c })(1-p^{-a-b-c} - p^{h_p(-b-c)} + p^{-a}p^{h_p(-b-c)})}{(1 - p^{-b- c})(1 - p^{-a-b-c})}
\\
=
\frac{1 - p^{-a-b-c} + p^{-b-c}[-1 + p^{-a-b-c} + (1-p^{c})(1-p^{-a-b-c} - p^{h_p(-b-c)} + p^{-a}p^{h_p(-b-c)})]}{(1 - p^{-b- c})(1 - p^{-a-b-c})}
\\
= \frac{1 - p^{-a-b-c } + p^{-b-c}[p^{-a-b} - p^{ c} + (1-p^c)p^{h_p(-b-c)}(-1 + p^{-a})]}{(1 - p^{-b- c })(1 - p^{-a-b-c})}
\\
= \frac{1 - p^{-b}(1+ p^{-a-c} - p^{-a-b-c}) + p^{-b}(1-p^{-a})(1-p^{-c})p^{h_p(-b-c)}}{(1 - p^{-b- c })(1 - p^{-a-b-c})}
\\
= \frac{(1 - p^{-b})(1- p^{-a-b-c}) + p^{-b}(1-p^{-a})(1-p^{-c})p^{h_p(-b-c)}}{(1 - p^{-b- c })(1 - p^{-a-b-c})}. 
\quad
\qedhere
\end{align}
\end{proof}
Applying Corollary \ref{coro:arithmeticfactor} to \eqref{eq:eart}, we obtain
\begin{multline}
I^{(1)}_{\alpha,\beta,\gamma,\delta} = \frac{\zeta(1- \alpha + \beta) \zeta(1 - \gamma + \delta)}{h^{\half - \alpha} k^{\half - \gamma}} 
\int_{-\infty}^{\infty} w(t) 
\\
\frac{1}{2 \pi i}  \int_{(1)}
\frac{\zeta(\alpha + \gamma +2s) \zeta(1 + \beta + \delta + 2s)}{\zeta(2-\alpha + \beta - \gamma + \delta)} 
C_{\alpha,\beta,\gamma,\delta,h,k}(s)
 \frac{G(s)}{s}  g(s,t)   \left(\frac{hk}{\pi^2}\right)^{s} 
\\
\Gamma(\alpha + \gamma + 2s) t^{-\alpha - \gamma - 2s} 2 \cos({\textstyle \frac{\pi }{2}}(\alpha + \gamma + 2s))\left(1 + O\left(\frac{1 + |s|^2}{t}\right)\right).
\end{multline}
Moving the line of integration back to $\varepsilon$ gives
\begin{multline}
I^{(1)}_{\alpha,\beta,\gamma,\delta} = \frac{\zeta(1- \alpha + \beta) \zeta(1 - \gamma + \delta)}{\zeta(2-\alpha + \beta - \gamma + \delta)}
\left[
\frac{1}{\sqrt{hk}} 
\int_{-\infty}^{\infty} t^{-\alpha - \gamma} w(t) 
\right.
\\
\frac{1}{2 \pi i}  \int_{(\varepsilon)}
\zeta(\alpha + \gamma +2s) \zeta(1 + \beta + \delta + 2s)
h^{\alpha} k^{\gamma} C_{\alpha,\beta,\gamma,\delta,h,k}(s)
 \frac{G(s)}{s}  g(s,t)   \left(\frac{hk}{\pi^2}\right)^{s} 
\\
\left. \Gamma(\alpha + \gamma + 2s) t^{- 2s} 2 \cos({\textstyle \frac{\pi }{2}}(\alpha + \gamma + 2s))
+ O\left(\frac{1}{\sqrt{hk}} (hkT)^{\varepsilon} \right)
\right].
\end{multline}
Applying the functional equation for the Riemann zeta function gives that
\begin{equation}
\label{eq:zetaFE}
\pi^{-2s} \zeta(\alpha + \gamma +2s) \Gamma(\alpha + \gamma + 2s)  2 \cos({\textstyle \frac{\pi}{2}(\alpha + \gamma + 2s)})
= \pi^{\alpha + \gamma} 2^{\alpha + \gamma + 2s} \zeta(1-\alpha-\gamma-2s).
\end{equation}
Using \eqref{eq:zetaFE} and Stirling's approximation for $g(s,t)$ gives
\begin{multline}
I^{(1)}_{\alpha,\beta,\gamma,\delta} =  \frac{\zeta(1- \alpha + \beta) \zeta(1 - \gamma + \delta)}{\zeta(2-\alpha + \beta - \gamma + \delta)} \left[ \frac{1}{\sqrt{hk}} 
\int_{-\infty}^{\infty} w(t) \left(\frac{t}{2 \pi}\right)^{-\alpha - \gamma} \right.
\\
\frac{1}{2 \pi i}  \int_{(\varepsilon)} \frac{G(s)}{s}  
\zeta(1-\alpha  -\gamma -2s) \zeta(1 + \beta + \delta + 2s) h^{\alpha} k^{\gamma} (hk)^s C_{\alpha,\beta,\gamma,\delta,h,k}(s)
ds dt
\\
\left. + O\left(\frac{1}{\sqrt{hk}} (hkT)^{\varepsilon} \right) \right],
\end{multline}
as claimed at the beginning of Section \ref{section:mainterms}.  Grouping the error terms together gives
\begin{myprop}
\label{prop:M}
We have
\begin{equation}
I_{O}^{(1)}(h,k) = M^{(1)}_{\alpha,\beta,\gamma,\delta} + M^{(1)}_{\beta, \alpha,\gamma,\delta} + M^{(1)}_{\alpha,\beta, \delta \gamma} + M^{(1)}_{\beta, \alpha,\delta, \gamma} + O(T^{\frac34 + \varepsilon} (hk)^{\frac78} (T/T_0)^{\frac94}),
\end{equation}
where
\begin{multline}
M^{(1)}_{\alpha,\beta,\gamma,\delta} =  \frac{\zeta(1- \alpha + \beta) \zeta(1 - \gamma + \delta)}{\zeta(2-\alpha + \beta - \gamma + \delta)}  \frac{1}{\sqrt{hk}} 
\int_{-\infty}^{\infty} w(t) \left(\frac{t}{2 \pi}\right)^{-\alpha - \gamma}
\\
\frac{1}{2 \pi i}  \int_{(\varepsilon)} \frac{G(s)}{s}  
\zeta(1-\alpha  -\gamma -2s) \zeta(1 + \beta + \delta + 2s) h^{\alpha} k^{\gamma} (hk)^s C_{\alpha,\beta,\gamma,\delta,h,k}(s)
ds dt.
\end{multline}
The estimate holds uniformly in terms of the shift parameters.
\end{myprop}
The uniformity follows because the sum of the $M$'s is holomorphic in terms of the shift parameters, and thus the error term must have the same property.

\subsection{A functional equation for $C(s)$}
It is a remarkable fact that $C_{\alpha,\beta,\gamma,\delta,h}(s)$ satisfies a functional equation relating $s$ and $-s$.
\begin{mytheo}
\label{thm:CFE}
We have
\begin{equation}
h^{-s+\alpha} C_{\alpha,\beta,\gamma,\delta,h}(-s) = h^{s-\delta} C_{-\delta,-\gamma,-\beta,-\alpha,h}(s).
\end{equation}
\end{mytheo}
We instantly deduce
\begin{mycoro}
\label{coro:CFE}
We have
\begin{equation}
(hk)^{-s} h^{\alpha} k^{\gamma} C_{\alpha,\beta,\gamma,\delta,h,k}(-s) = (hk)^{s} h^{-\delta} k^{-\beta} C_{-\delta,-\gamma,-\beta,-\alpha,h,k}(s)
\end{equation}
\end{mycoro}
\begin{proof}
It suffices to check the formula at each prime dividing $h$.  
Note that the functional equation above is equivalent to $C_{\alpha,\beta,\gamma,\delta,h}(-s) = h^{-\alpha - \delta + 2s}C_{-\delta,-\gamma, -\beta,-\alpha}(s)$.  It suffices to show that
\begin{equation}
\frac{C_{\alpha,\beta,\gamma,\delta,h}^{(i)}(-s)}{1- p^{-\alpha-\delta + 2s}}  = p^{(-\alpha - \delta + 2s)h_p} \frac{C_{-\delta,-\gamma,-\beta,-\alpha,h}^{(i)}(s)}{1- p^{\alpha+\delta - 2s}},
\end{equation}
for $i=0,1,2$.  Letting $x = p^{-\alpha - \delta +2s}$, and $h= h_p$, the three desired identities are
\begin{align}
\frac{1-x^{1 + h}}{1-x} &= x^{h} \frac{1- \frac{1}{x^{1+h}}}{1-\frac{1}{x}}, \\
\frac{ p^{\gamma - \delta} + p^{\alpha - \beta}x}{1-x} (1- x^h) &= x^h \frac{p^{-\beta + \alpha} + p^{-\delta + \gamma} \frac{1}{x}}{1 - \frac{1}{x}} (1- \frac{1}{x^h}), \\
\frac{x-x^h}{1-x} & = x^h \frac{\frac{1}{x} - \frac{1}{x^h}}{1 - \frac{1}{x}},
\end{align}
each of which is easily checked by inspection.
\end{proof}

\subsection{Combining terms}
It turns out that $M_{\alpha,\beta,\gamma,\delta}^{(1)}$ is considerably simplified when it is added to a corresponding term from the other part of the approximate functional equation.  It is easy to see that an analog of Proposition \ref{prop:M} holds for $I_{O}^{(2)}$, where each main term has the same form as that for $I_{O}^{(1)}$ but with $\alpha$ switched with $-\gamma$ and $\beta$ switched with $-\delta$, and multiplied by $X_{\alpha,\beta,\gamma,\delta,t} \sim (t/2\pi)^{-\alpha - \beta -\gamma - \delta}$.  The term to sum with $M_{\alpha,\beta,\gamma,\delta}^{(1)}$ is the one obtained by so modifying $M_{\beta,\alpha,\delta,\gamma}^{(1)}$, which we denote $M_{\alpha,\beta,\gamma,\delta}^{(2)}$.  
Hence we have
\begin{myprop}
\label{prop:deltaconclusion2}
We have
\begin{equation}
I_O^{(2)}(h,k) = M_{\alpha,\beta,\gamma,\delta}^{(2)} + M_{\beta, \alpha, \gamma,\delta}^{(2)} + M_{\alpha,\beta ,\delta, \gamma}^{(2)} + M_{\beta, \alpha, \delta, \gamma}^{(2)} + 
O(T^{\frac34 + \varepsilon} (hk)^{\frac78} (T/T_0)^{\frac94}),
\end{equation}
where
\begin{multline}
M_{\alpha,\beta,\gamma,\delta}^{(2)} =  \frac{\zeta(1- \alpha + \beta) \zeta(1 - \gamma + \delta)}{\zeta(2-\alpha + \beta - \gamma + \delta)} \frac{1}{\sqrt{hk}} 
\int_{-\infty}^{\infty} w(t) \left(\frac{t}{2 \pi}\right)^{-\alpha - \gamma} 
\\
\frac{1}{2 \pi i}  \int_{(\varepsilon)} \frac{G(s)}{s}  
\zeta(1-\alpha  -\gamma +2s) \zeta(1 + \beta + \delta - 2s) h^{-\delta} k^{-\beta} (hk)^s C_{-\delta,-\gamma,-\beta,-\alpha,h,k}(s)
ds dt
\end{multline}
\end{myprop}

The result of adding $M_{\alpha,\beta,\gamma,\delta}^{(1)}$ and $M_{\alpha,\beta,\gamma,\delta}^{(2)}$ is given by the following
\begin{myprop}
We have
\begin{equation}
M_{\alpha,\beta,\gamma,\delta}^{(1)} + M_{\alpha,\beta,\gamma,\delta}^{(2)} = P^{(0)}_{\alpha,\beta,\gamma,\delta} + P^{(1)}_{\alpha,\beta,\gamma,\delta} + P^{(2)}_{\alpha,\beta,\gamma,\delta},
\end{equation}
where
\begin{equation}
\label{eq:plasticbag}
P^{(0)}_{\alpha,\beta,\gamma,\delta} = \frac{1}{\sqrt{hk}} \int_{-\infty}^{\infty} w(t) \left(\frac{t}{2 \pi}\right)^{-\alpha - \gamma} A_{-\gamma,\beta,-\alpha,\delta}(0) 
h^{\alpha} k^{\gamma} C_{\alpha,\beta,\gamma,\delta,h,k}(0)
dt,
\end{equation}
\begin{multline}
P^{(1)}_{\alpha,\beta,\gamma,\delta}   = - \half \frac{\zeta(1- \alpha + \beta) \zeta(1 - \gamma + \delta) \zeta(1 -\alpha + \beta -\gamma + \delta )}{\zeta(2-\alpha + \beta - \gamma + \delta)} \frac{1}{\sqrt{hk}} 
\\
\int_{-\infty}^{\infty} w(t) \left(\frac{t}{2 \pi}\right)^{-\alpha - \gamma}
\frac{G(\frac{-\alpha - \gamma}{2})}{\frac{-\alpha - \gamma}{2}}  
 (h/k)^{\frac{\alpha - \gamma}{2}} C_{\alpha,\beta,\gamma,\delta,h,k}\left(\frac{-\alpha - \gamma}{2}\right)
 dt,
\end{multline}
and
\begin{multline}
P^{(2)}_{\alpha,\beta,\gamma,\delta}   = \half \frac{\zeta(1- \alpha + \beta) \zeta(1 - \gamma + \delta) \zeta(1 -\alpha + \beta -\gamma + \delta )}{\zeta(2-\alpha + \beta - \gamma + \delta)} \frac{1}{\sqrt{hk}} 
\\
\int_{-\infty}^{\infty} w(t) \left(\frac{t}{2 \pi}\right)^{-\alpha - \gamma}
\frac{G(\frac{-\beta - \delta}{2})}{\frac{-\beta - \delta}{2}}  
h^{\alpha} k^{\gamma} (hk)^{\frac{-\beta - \delta}{2}}  C_{\alpha,\beta,\gamma,\delta,h,k}\left(\frac{-\beta - \delta}{2}\right) dt.
\end{multline}
\end{myprop}
\begin{proof}
We begin by developing $M^{(1)}$ by moving its line of integration to $-\varepsilon$, passing poles at $s=0$, $2s = -\alpha - \gamma$, and $2s = -\beta - \delta$.  The pole at $s=0$ gives
\begin{multline}
\frac{\zeta(1- \alpha + \beta) \zeta(1 - \gamma + \delta) \zeta(1-\alpha  -\gamma) \zeta(1 + \beta + \delta)}{\zeta(2-\alpha + \beta - \gamma + \delta)} 
\\
\frac{1}{\sqrt{hk}} 
\int_{-\infty}^{\infty} w(t) \left(\frac{t}{2 \pi}\right)^{-\alpha - \gamma}
h^{\alpha} k^{\gamma} C_{\alpha,\beta,\gamma,\delta,h,k}(0)
dt,
\end{multline}
which is precisely \eqref{eq:plasticbag} (recall $A_{\alpha,\beta,\gamma,\delta}(s)$ was given by \eqref{eq:Adef}).  The pole at $2s= -\alpha - \gamma$ gives $P^{(1)}$ and the pole at $2s = -\beta - \delta$ gives $P^{(2)}$.  

On the new line apply the change of variable $s \rightarrow -s$.  The functional equation for $C(s)$ as stated in Corollary \ref{coro:CFE} shows that the new integral cancels $M^{(2)}_{\alpha,\beta,\gamma,\delta}$.
\end{proof}

Our next goal is to show
\begin{myprop}
\label{prop:P}
We have
\begin{equation}
\label{eq:P0}
P^{(0)}_{\alpha,\beta,\gamma,\delta} =  \frac{1}{\sqrt{hk}} \int_{-\infty}^{\infty} \left(\frac{t}{2 \pi}\right)^{-\alpha - \gamma} w(t) Z_{-\gamma,\beta,-\alpha,\delta,h,k}(0) dt,
\end{equation}
and
\begin{equation}
\label{eq:Pi}
P^{(i)}_{\alpha,\beta,\gamma,\delta} = - J^{(i)}_{\alpha,\beta,\gamma,\delta},
\end{equation}
for $i=1,2$.
\end{myprop}
This proposition will complete the proof of Proposition \ref{prop:offdiagonal}, by summing over the four permutations of the shift parameters.  We shall prove Proposition \ref{prop:P} at the end of Section \ref{section:Zdev}.  We need to exhibit a relation between $C_{\alpha,\beta,\gamma,\delta,h,k}(s)$ and $Z_{\alpha,\beta,\gamma,\delta,h,k}(s)$.

\subsection{Development of $Z(s)$}
\label{section:Zdev}
We require a different formulation of $Z(s)$ in order to recognize relations with $C(s)$.  Let
\begin{equation}
B_{\alpha,\beta,\gamma,\delta,h}(s) = \prod_{p^{h_p} || h} \left(\frac{\sum_{j=0}^{\infty} \sigma_{\alpha,\beta}(p^j) \sigma_{\gamma,\delta}(p^{j + h_p}) p^{-j(s+1)}}{\sum_{j=0}^{\infty} \sigma_{\alpha,\beta}(p^j) \sigma_{\gamma,\delta}(p^{j}) p^{-j(s+1)}} 
\right),
\end{equation}
so that
\begin{equation}
B_{\alpha,\beta,\gamma,\delta,h,k}(s) = B_{\alpha,\beta,\gamma,\delta,h}(s) B_{\gamma,\delta,\alpha,\beta,k}(s).
\end{equation}
Recall $B_{\alpha,\beta,\gamma,\delta,h,k}(s)$ was defined by \eqref{eq:B}. 

\begin{mylemma}
We have
\begin{equation}
B_{\alpha,\beta,\gamma,\delta,h}(s) = \prod_{p^{h_p} || h} \left(\frac{B^{(0)}(s) - p^{-1} B^{(1)}(s)  + p^{-2}B^{(2)}(s) }{(p^{-\gamma} - p^{-\delta})(1  - p^{-2-\alpha - \beta - \gamma - \delta}p^{-2s})} \right),
\end{equation}
where
\begin{align}
B^{(0)}(s)  =  B^{(0)}_{\alpha,\beta,\gamma,\delta,h}(s) & = p^{-\gamma(1+h_p)} - p^{-\delta(1 + h_p)}, \\
B^{(1)}(s)  =  B^{(1)}_{\alpha,\beta,\gamma,\delta,h}(s) & = (p^{-\alpha} + p^{-\beta})p^{-\gamma - \delta} (p^{-\gamma h_p} - p^{-\delta h_p})p^{-s}, \\
B^{(2)}(s)  =  B^{(2)}_{\alpha,\beta,\gamma,\delta,h,2}(s) & =  p^{-\alpha - \beta - \gamma - \delta}(p^{-\delta - \gamma h_p} - p^{-\gamma - \delta h_p})p^{-2s}.
\end{align}
\end{mylemma}
\begin{proof}
A preliminary step is to compute
\begin{equation}
\sum_{j=0}^{\infty} \sigma_{\alpha,\beta}(p^j) \sigma_{\gamma,\delta}(p^{j + h_p}) p^{-j(s+1)}.
\end{equation}
Recall
\begin{align}
\sigma_{\alpha,\beta}(p^{m}) = \sum_{n_1 n_2 = p^m} n_1^{-\alpha} n_2^{-\beta} = \sum_{0 \leq j \leq m} p^{-j\alpha} p^{-(m-j)\beta} = p^{-m\beta} \sum_{0 \leq j \leq m} p^{(-\alpha + \beta)j} 
\\
= p^{-m\beta} \frac{1 - p^{(-\alpha + \beta)(m+1)}}{1 - p^{-\alpha + \beta}}
= \frac{p^{-(m+1)\alpha} - p^{-(m+1)\beta}}{p^{-\alpha} - p^{-\beta}}.
\end{align}
Thus
\begin{align}
\sum_{j=0}^{\infty} \sigma_{\alpha,\beta}(p^j) \sigma_{\gamma,\delta}(p^{j + h_p}) p^{-j(s+1)} &=
\\ 
\sum_{j=0}^{\infty} & \frac{(p^{-(j+1)\alpha} - p^{-(j+1)\beta} )
(p^{-(j + h_p+1)\gamma} - p^{-(j + h_p+1)\delta})}{(p^{-\alpha} - p^{-\beta})(p^{-\gamma} - p^{-\delta})}p^{-j(s+1)},
\end{align}
which upon multiplying out is the sum of four terms.  One of them is
\begin{align}
\sum_{j=0}^{\infty} \frac{(p^{-(j+1)\alpha} )
(p^{-(j + h_p+1)\gamma} )}{(p^{-\alpha} - p^{-\beta})(p^{-\gamma} - p^{-\delta})}p^{-j(s+1)} = \frac{p^{-\alpha - \gamma(1 + h_p)}}{(p^{-\alpha} - p^{-\beta})(p^{-\gamma} - p^{-\delta})} \sum_{j=0}^{\infty} \frac{1}{p^{j(1+s + \alpha + \gamma)}}
\\
= \frac{p^{-\alpha - \gamma(1 + h_p)} }{(p^{-\alpha} - p^{-\beta})(p^{-\gamma} - p^{-\delta})(1 - p^{-1-s - \alpha - \gamma})}.
\end{align}
The other three terms are gotten by switching $\alpha$ and $\beta$ or $\gamma$ and $\delta$, and so appropriately summing them gives
\begin{multline}
\sum_{j=0}^{\infty} \sigma_{\alpha,\beta}(p^j) \sigma_{\gamma,\delta}(p^{j + h_p}) p^{-j(s+1)} = (p^{-\alpha} - p^{-\beta})^{-1} (p^{-\gamma} - p^{-\delta})^{-1}
\\
\times \left( \frac{p^{-\alpha - \gamma(1 + h_p)}}{1 - p^{-1-s - \alpha - \gamma}} 
- \frac{p^{-\beta - \gamma(1 + h_p)}}{1 - p^{-1-s - \beta - \gamma}}
- \frac{p^{-\alpha - \delta(1 + h_p)}}{1 - p^{-1-s - \alpha - \delta}}
+ \frac{p^{-\beta - \delta(1 + h_p)}}{1 - p^{-1-s - \beta - \delta}}
\right),
\end{multline}
which simplifies to
\begin{equation}
\frac{1}{p^{-\gamma} - p^{-\delta}} 
\left( \frac{p^{- \gamma(1 + h_p)}}{(1 - p^{-1-s - \alpha - \gamma})(1 - p^{-1-s - \beta - \gamma})} 
- \frac{p^{-\delta(1 + h_p)}}{(1 - p^{-1-s - \alpha - \delta})(1 - p^{-1-s - \beta - \delta})}
\right),
\end{equation}
and expands further into
\begin{multline}
\frac{B^{(0)}(s) - p^{-1} B^{(1)}(s) + p^{-2} B^{(2)}(s)}{(p^{-\gamma} - p^{-\delta})(1 - p^{-1-s - \alpha - \gamma})(1 - p^{-1-s - \beta - \gamma})(1 - p^{-1-s - \alpha - \delta})(1 - p^{-1-s - \beta - \delta})}.
\end{multline}

We conclude that
\begin{equation}
\frac{\sum_{j=0}^{\infty} \sigma_{\alpha,\beta}(p^j) \sigma_{\gamma,\delta}(p^{j + h_p}) p^{-j(s+1)}}{\sum_{j=0}^{\infty} \sigma_{\alpha,\beta}(p^j) \sigma_{\gamma,\delta}(p^{j}) p^{-j(s+1)}}
= 
\frac{B^{(0)}(s) - p^{-1} B^{(1)}(s) + p^{-2} B^{(2)}(s)}
{(p^{-\gamma} - p^{-\delta})(1  - p^{-\alpha - \beta - \gamma - \delta}p^{-2s})},
\end{equation}
which completes the proof.
\end{proof}

We have
\begin{mylemma}
\label{lemma:CB0}
We have
\begin{equation}
h^{\alpha} k^{\gamma} C_{\alpha,\beta,\gamma,\delta,h,k}(0) = B_{-\gamma,\beta,-\alpha,\delta,h,k}(0).
\end{equation}
\end{mylemma}
\begin{proof}
By symmetry, it suffices to show
\begin{equation}
h^{\alpha} C_{\alpha,\beta,\gamma,\delta,h}(0) = B_{-\gamma,\beta,-\alpha,\delta,h}(0).
\end{equation}

Inspection of the Euler products for $B$ and $C$ reduces the problem to showing
\begin{equation}
\frac{p^{\alpha h_p} C_{\alpha,\beta,\gamma,\delta,h}^{(i)}(0)}{(1 - p^{-\alpha - \delta})} = \frac{B_{-\gamma,\beta,-\alpha,\delta,h}^{(i)}(0)}{p^{\alpha} - p^{-\delta}},
\end{equation}
which reduces to showing
\begin{equation}
p^{\alpha h_p} C_{\alpha,\beta,\gamma,\delta,h}^{(i)}(0) = p^{-\alpha} B_{-\gamma,\beta,-\alpha,\delta,h}^{(i)}(0)
\end{equation}
for $i=0,1,2$.  These identities are
\begin{align}
p^{\alpha h_p} (1- p^{(-\alpha - \delta)(1 + h_p)}) &= p^{-\alpha} (p^{\alpha(1 +h_p)} - p^{-\delta(1 + h_p)}), \\
p^{\alpha h_p} (p^{\gamma - \delta} + p^{\alpha - \beta} p^{-\alpha - \delta})(1- p^{(-\alpha - \delta)h_p}) &= p^{-\alpha} (p^{\gamma} + p^{-\beta}) p^{\alpha - \delta} (p^{\alpha h_p} - p^{-\delta h_p}), \\
p^{\alpha h_p} p^{\alpha - \beta + \gamma - \delta} (p^{-\alpha - \delta} - p^{(-\alpha- \delta)h_p}) & = p^{-\alpha} p^{\gamma - \beta + \alpha - \delta} (p^{-\delta + \alpha h_p} - p^{\alpha - \delta h_p}),
\end{align}
each of which is easily checked by inspection.
\end{proof}
\begin{mylemma}
\label{lemma:CB1}
We have
\begin{multline}
\half \frac{\zeta(1-\alpha + \beta) \zeta(1-\gamma + \delta) \zeta(1-\alpha + \beta - \gamma + \delta)}{\zeta(2-\alpha + \beta -\gamma + \delta)} \left(\frac{h}{k}\right)^{\frac{\alpha - \gamma}{2}} C_{\alpha,\beta,\gamma,\delta,h,k}\left(\frac{-\alpha - \gamma}{2}\right) 
\\
= \text{Res}_{s=\frac{-\alpha-\gamma}{2}} (Z_{\alpha,\beta,\gamma, \delta,h,k}(2s)) (hk)^{\frac{\alpha + \gamma}{2}}.
\end{multline}
\end{mylemma}
\begin{proof}
First, note that
\begin{align}
\text{Res}_{s=\frac{-\alpha-\gamma}{2}} Z_{\alpha,\beta,\gamma, \delta,h,k}(2s)  = 
\text{Res}_{s=\frac{-\alpha-\gamma}{2}}(A_{\alpha,\beta,\gamma,\delta}(2s)) B_{\alpha,\beta,\gamma,\delta,h,k}(-\alpha-\gamma)
\\
=\half  \frac{\zeta(1-\alpha + \beta) \zeta(1-\gamma + \delta) \zeta(1-\alpha + \beta - \gamma + \delta)}{\zeta(2-\alpha + \beta -\gamma + \delta)} B_{\alpha,\beta,\gamma,\delta,h,k}(-\alpha-\gamma).
\end{align}

Thus it suffices to check
\begin{equation}
\label{eq:CB}
h^{-\gamma} C_{\alpha,\beta,\gamma,\delta,h}\left(\frac{-\alpha - \gamma}{2}\right) = B_{\alpha,\beta,\gamma,\delta,h}(-\alpha-\gamma),
\end{equation}
since the analogous identity for the primes dividing $k$ follows by symmetry.  Again, inspection of the Euler products reduces the problem to showing
\begin{equation}
\frac{p^{-\gamma h_p} C^{(i)}_{\alpha,\beta,\gamma,\delta,h}\left(\frac{-\alpha-\gamma}{2}\right)}{(1 - p^{-2 + \alpha - \beta + \gamma - \delta})(1 - p^{-\alpha - \delta + \alpha + \gamma})} = \frac{B^{(i)}_{\alpha,\beta,\gamma,\delta}(-\alpha - \gamma)}{(p^{-\gamma} - p^{-\delta})(1 - p^{-2 -\alpha - \beta - \gamma - \delta + 2\alpha + 2\gamma})},
\end{equation}
which reduces to showing
\begin{equation}
p^{-\gamma h_p} C_{\alpha,\beta,\gamma,\delta,h}^{(i)}\left(\frac{-\alpha - \gamma}{2}\right) = p^{\gamma} B_{\alpha,\beta,\gamma,\delta,h}^{(i)}(-\alpha-\gamma),
\end{equation}
for $i=0,1,2$.  These identities
are equivalent to
\begin{align}
p^{-\gamma h_p} (1 - p^{(\gamma-\delta)(1+h_p)}) & = p^{\gamma} (p^{-\gamma(1+h_p)} - p^{-\delta (1+h_p)}), \\
p^{-\gamma h_p} (p^{\gamma-\delta} + p^{\alpha-\beta} p^{\gamma-\delta}) (1- p^{(\gamma-\delta)h_p}) & = p^{\gamma} (p^{-\alpha} + p^{-\beta}) p^{-\gamma-\delta} (p^{-\gamma h_p} - p^{-\delta h_p}) p^{\alpha + \gamma}, \\
p^{-\gamma h_p} p^{\alpha-\beta+\gamma-\delta} (p^{\gamma-\delta} - p^{(\gamma-\delta)h_p)}) &= p^{\gamma} p^{-\alpha - \beta - \gamma - \delta} (p^{-\delta - \gamma h_p} - p^{-\gamma - \delta h_p}) p^{2\alpha + 2\gamma},
\end{align}
each of which is established by inspection.
\end{proof}
\begin{mylemma}
\label{lemma:CB2}
We have
\begin{multline}
\half \frac{\zeta(1- \alpha + \beta) \zeta(1 - \gamma + \delta) \zeta(1 -\alpha + \beta -\gamma + \delta )}{\zeta(2-\alpha + \beta - \gamma + \delta)} 
h^{\alpha} k^{\gamma} C_{\alpha,\beta,\gamma,\delta,h,k}\left(\frac{-\beta - \delta}{2}\right)
\\
=\text{Res}_{s=\frac{\beta + \delta}{2}}(Z_{-\gamma,-\delta,-\alpha,-\beta,h,k}(2s)).
\end{multline}
\end{mylemma}
\begin{proof}
We begin by noting that
\begin{multline}
\text{Res}_{s=\frac{\beta + \delta}{2}}(Z_{-\gamma,-\delta,-\alpha,-\beta,h,k}(2s)) 
\\
= \half \frac{\zeta(1- \alpha + \beta) \zeta(1 - \gamma + \delta) \zeta(1 -\alpha + \beta -\gamma + \delta )}{\zeta(2-\alpha + \beta - \gamma + \delta)} B_{-\gamma,-\delta,-\alpha,-\beta,h,k}(\beta + \delta),
\end{multline}
so it suffices to show
\begin{equation}
h^{\alpha} k^{\gamma} C_{\alpha,\beta,\gamma,\delta,h,k}\left(\frac{-\beta - \delta}{2}\right) = B_{-\gamma,-\delta,-\alpha,-\beta,h,k}(\beta + \delta),
\end{equation}
which reduces to showing
\begin{equation}
h^{\alpha}  C_{\alpha,\beta,\gamma,\delta,h}\left(\frac{-\beta - \delta}{2}\right) = B_{-\gamma,-\delta,-\alpha,-\beta,h}(\beta + \delta),
\end{equation}
by symmetry.  Applying Theorem \ref{thm:CFE} and then \eqref{eq:CB} gives
\begin{align}
h^{\alpha}  C_{\alpha,\beta,\gamma,\delta,h}\left(\frac{-\beta - \delta}{2}\right) = h^{\beta} C_{-\delta,-\gamma, -\beta,-\alpha,h}\left(\frac{\beta + \delta}{2}\right)
\\
= B_{-\delta,-\gamma, -\beta,-\alpha,h}\left(\beta + \delta\right),
\end{align}
which equals $B_{-\gamma,-\delta,-\alpha,-\beta,h}(\beta + \delta)$ by inspection of \eqref{eq:B}.
\end{proof}

\begin{proof}[Proof of Proposition \ref{prop:P}]
Using Lemma \ref{lemma:CB0} gives \eqref{eq:P0}, and Lemmas \ref{lemma:CB1} and \ref{lemma:CB2} give \eqref{eq:Pi} for $i=1$ and $2$, respectively.
\end{proof}

\section{Twisted moment conjectures}
The recent five author paper \cite{CFKRS05} has produced a recipe that can conjecture the asymptotics for the integral moments of a family of $L$-functions.  Their recipe does not directly apply to twisted moments of the form considered in this paper, so it is perhaps not clear how to guess what the answer should be.  In fact, we predict that a rather simple modification of their recipe can be used to conjecture the form of the asymptotics.  For simplicity, we consider the twisted $2\ell$-th moment of the Riemann zeta function, that is
\begin{multline}
I_{2\ell}(h,k) \\ := \int_{-\infty}^{\infty} 
\left(\frac{h}{k}\right)^{-it}
\zeta({\textstyle \half + \alpha_1 + it}) \dots
\zeta({\textstyle \half + \alpha_\ell + it}) 
\zeta({\textstyle \half + \beta_1 - it}) \dots 
\zeta({\textstyle \half + \beta_\ell - it}) w(t)dt,
\end{multline}
where $w$ is a nice function with support in $[T/2, 4T]$ and $(h,k) = 1$.  Following the derivation of the moment conjecture of Section 2.2 of \cite{CFKRS05} (that is, the case $h=k=1$), for each occurence of $\zeta$ we write
\begin{equation}
\zeta(\tfrac12 + s) = \sum n^{-\half - s} + \chi(\thalf + s) \sum n^{-\half +s}
\end{equation}
and multiply out the various sums, obtaining $2^{2 \ell}$ terms.  Now we throw away the terms where the product of $\chi$-factors is oscillatory, which amounts to keeping the terms with an equal number of $\chi(\thalf + \alpha_i + it)$ and $\chi(\thalf + \beta_j - it)$ terms.  Note that by Stirling's approximation,
\begin{equation}
\chi(\thalf + \alpha + it) \chi(\thalf + \beta - it) \sim \left(\frac{t}{2\pi}\right)^{-\alpha-\beta}.
\end{equation}
Now consider the term where we take the first part of the approximate functional equation for each occurence of $\zeta$, namely
\begin{equation}
\int_{-\infty}^{\infty} 
\left(\frac{h}{k}\right)^{-it} \sum_{\substack{m_1,\dots, m_{\ell} \\ n_1, \dots, n_{\ell}}} \frac{1}{m_1^{\half + \alpha_1} \dots m_\ell^{\half + \alpha_\ell} n_1^{\half + \beta_1} \dots n_\ell^{\half + \beta_\ell}} \left(\frac{m_1 \dots m_\ell}{n_1 \dots n_\ell}\right)^{-it} w(t) dt.
\end{equation}
In the averaging we only retain the diagonal term $h m_1 \dots m_\ell = k n_1 \dots n_\ell$, which is
\begin{equation}
\frac{1}{\sqrt{hk}} \int_{-\infty}^{\infty} w(t) Z_{\alpha, \beta}(0)dt,
\end{equation}
where
\begin{align}
Z_{\alpha, \beta}(s) := (hk)^{\half + s} \sum_{h m_1 \dots m_\ell = k n_1 \dots n_\ell} \frac{1}{m_1^{\half + \alpha_1 + s} m_{\ell}^{\half + \alpha_\ell + s} n_1^{\half + \beta_1 + s} n_\ell^{\half + \beta_\ell + s}}
\\
= (hk)^{\half + s} \sum_{hm = kn} \frac{\sigma_{\alpha}(m) \sigma_{\beta}(n)}{m^{\half + s} n^{\half  + s}},
\end{align}
where
\begin{equation}
\sigma_{\alpha}(m) = \sigma_{\alpha_1,\dots, \alpha_k}(m) = \sum_{c_1 \dots c_k = m} c_1^{-\alpha_1} \dots c_k^{-\alpha_k}.
\end{equation}
Since $(h,k) = 1$, we get all solutions to $hm = kn$ by $m = k l$, $n = hl$, giving
\begin{equation}
Z_{\alpha, \beta}(s) = \sum_l \frac{\sigma_{\alpha}(kl) \sigma_{\beta}(hl)}{l^{1 + 2s}}.
\end{equation}
Let
\begin{equation}
A_{\alpha,\beta}(s) = \sum_l \frac{\sigma_{\alpha}(l) \sigma_{\beta}(l)}{l^{1 + 2s}},
\end{equation}
suppose $p^{h_p} || h$ and $p^{k_p} || k$, and set
\begin{multline}
B_{\alpha,\beta,h,k}(s) = \prod_{p | h} \left(\frac{\sum_{j=0}^{\infty} \sigma_{\alpha}(p^j) \sigma_{\beta}(p^{j + h_p}) p^{-j(s+1)}}{\sum_{j=0}^{\infty} \sigma_{\alpha}(p^j) \sigma_{\beta}(p^{j}) p^{-j(s+1)}} 
\right)
\\
\times \prod_{p | k} 
\left(
\frac{\sum_{j=0}^{\infty} \sigma_{\alpha}(p^{j + k_p}) \sigma_{\beta}(p^{j}) p^{-j(s+1)}}{\sum_{j=0}^{\infty} \sigma_{\alpha}(p^j) \sigma_{\beta}(p^{j}) p^{-j(s+1)} }
\right),
\end{multline}
so that
\begin{equation}
Z_{\alpha,\beta}(s) =  A_{\alpha,\beta}(s) B_{\alpha,\beta,h,k}(2s).
\end{equation}
Extracting the polar behavior of $A_{\alpha,\beta}(s)$ near $s=0$ gives
\begin{equation}
A_{\alpha,\beta}(s) = \left( \prod_{1\leq i ,j \leq \ell} \zeta(1 + s + \alpha_i + \beta_j) \right) \widetilde{A}_{\alpha,\beta}(s),
\end{equation}
where $\widetilde{A}_{\alpha,\beta}(s)$ is given by an Euler product that is absolutely convergent in some half plane $\text{Re}(s) > -\delta$ for some $\delta > 0$ (depending on $\ell$).  Thus we obtain the meromorphic continuation of $Z_{\alpha,\beta}(s)$ to $s=0$ via
\begin{equation}
Z_{\alpha,\beta}(0) = \left( \prod_{1\leq i ,j \leq \ell} \zeta(1 + \alpha_i + \beta_j) \right) \widetilde{A}_{\alpha,\beta}(0) B_{\alpha,\beta,h,k}(0).
\end{equation}

We are left with a term generalizing the $Z_{\alpha,\beta,\gamma,\delta,h,k}(0)$ term appearing in Theorem \ref{thm:mainresult} (in case $\ell=2$ it is precisely the same).  The final conjecture is obtained by summing over the $\binom{2 \ell}{\ell}$ permutations gotten by swapping an equal number of $\alpha_i$'s and $-\beta_j$'s, and for each such swap, multiplying by $(t/2\pi)^{-\alpha_i - \beta_j}$.  This procedure is an obvious generalization of the way to write the main term of our Theorem \ref{thm:mainresult}.  Explicitly, we write
\begin{myconj}
Let $\Phi_j$ be the set of subsets of $\{\alpha_1, \dots, \alpha_\ell\}$ of cardinality $j$, for $j=0,\dots, \ell$, and similarly let $\Psi_j$ be the set of subsets of $\{\beta_1, \dots, \beta_\ell\}$ of cardinality $j$.  If $S \in \Phi_j$ and $T \in \Psi_j$ then write $S= \{\alpha_{i_1}, \dots, \alpha_{i_j}\}$ and $T=\{\beta_{l_1}, \dots, \beta_{l_j}\}$ where $i_1 < i_2 < \dots < i_j$ and $l_1 < l_2 < \dots < l_j$.  Let $(\alpha_S;\beta_T)$ be the tuple obtained from $(\alpha_1, \dots, \alpha_\ell;\beta_1, \dots, \beta_\ell)$ by replacing $\alpha_{i_r}$ with $-\beta_{i_r}$ and replacing $\beta_{i_r}$ with $-\alpha_{i_r}$ for $1 \leq r \leq j$.
We then conjecture that
\begin{equation}
I_{2\ell}(h,k) = \int_{-\infty}^{\infty} w(t) \left(\sum_{0 \leq j \leq \ell} \sum_{\substack{S \in \Phi_j \\ T \in \Psi_j}} Z_{\alpha_S;\beta_T}(0) \left(\frac{t}{2\pi}\right)^{-S-T}  + O(t^{-\half + \varepsilon})\right)dt,
\end{equation}
provided $hk \leq T^{\half - \varepsilon}$,
where we have written $(t/2 \pi)^{-S - T}$ for $(t/2\pi)^{-\sum_{x \in S} x - \sum_{y \in T} y}$.
\end{myconj}

Although we have only described the modified recipe for the zeta function in $t$-aspect, we also predict that an analogous modification of the recipe of \cite{CFKRS05} can be used to obtain conjectures for a general twisted moment for any family of $L$-functions (here twisting should be construed to mean multiplying by an appropriate harmonic).

\end{document}